\documentclass[bj,preprint]{imsart}

\RequirePackage[OT1]{fontenc}
\RequirePackage{amsthm,amsmath}
\RequirePackage[numbers]{natbib}
\RequirePackage[colorlinks,citecolor=blue,urlcolor=blue]{hyperref}
\usepackage{amsmath, amsthm, amsfonts,amssymb}
\usepackage{bm}
\usepackage{graphicx, color}
\usepackage{subfigure}
\usepackage{epstopdf}
\usepackage{float}
\usepackage{array}
\usepackage{booktabs}
\usepackage{multirow}

\usepackage{algorithm}
\usepackage{algpseudocode}
\DeclareMathOperator*{\argminB}{argmin}
\DeclareMathOperator*{\argmaxB}{argmax}

\startlocaldefs
\numberwithin{equation}{section}
\theoremstyle{plain}
\newtheorem{thm}{Theorem}[section]
\theoremstyle{remark}
\newtheorem{exam}[thm]{Example}
\newtheorem{rem}[thm]{Remark}
\newtheorem{prop}[thm]{Proposition}
\newtheorem{assm}[thm]{Assumption}
\newtheorem{defn}[thm]{Definition}
\newtheorem{lem}[thm]{Lemma}
\endlocaldefs

\begin{document}

\begin{frontmatter}
\title{High dimensional deformed rectangular matrices with applications in matrix denoising}
\runtitle{High dimensional matrix denoising }

\begin{aug}
\author{\fnms{Xiucai} \snm{Ding}\ead[label=e1]{xiucai.ding@mail.utoronto.ca.}}

\runauthor{Xiucai. Ding}

\affiliation{University of Toronto}

\address{Department of Statistics, 
University of Toronto, 
Toronto, Ontario, M5S 3G3, 
Canada\\
\printead{e1}\\
\phantom{E-mail:\ }
\printead*{}}
\end{aug}

\begin{abstract} 
We consider the recovery of a low rank $M \times N$ matrix $S$ from its noisy observation $\tilde{S}$ in the high dimensional framework when $M$ is comparable to $N$.  We propose two efficient estimators for $S$ under two different regimes. Our analysis relies on the local asymptotics of the eigenstructure of large dimensional rectangular matrices with finite rank perturbation. We derive the convergent limits and rates for the singular values and vectors for such matrices.    
\end{abstract}

%\begin{keyword}[class=MSC]
%\kwd{15B52}
%\kwd{60B20}
%\end{keyword}


\end{frontmatter}

%
%{\color{red} \bf One possibility:  we can add a corollary and appendix (discussion is fine and refer the details to our new paper) to discuss when the covariance structure of $X$ is of general bulk form, where the convergent limits in the paper can be dependent on $N.$ We need to discuss more on the anisotropic law (local laws with deterministic convergent limit depending on $N.$)}
\section{Introduction}
Matrix denoising is important in many scientific endeavors. They appear prominently in singal processing \cite{TA}, image denoising \cite{ME}, machine learning \cite{YMB}, statistics \cite{GD2, GD1,JWHT}, empirical finance \cite{BBP2017, LCPB} and biology \cite{PGSR}. In these applications, researchers are interested in recovering the true deterministic matrix from a noisy observation. Consider that we can observe a noisy $M \times N$ data matrix $\tilde{S}_N$, where
\begin{equation}\label{defn_m11}
\tilde{S}_N=S_N+X_N,
\end{equation}
the deterministic matrix $S_N$ is known as the signal matrix and $X_N$ the noise matrix.  In the classic framework where $M$ is much smaller than $N,$ the truncated singular value decomposition (TSVD) is the default technique, see for example \cite{GVL}.  This method recovers $S_N$ with an estimator $ \hat{S}_N=\sum_{i=1}^m \mu_i \tilde{u}_i \tilde{v}_i^*$ using the truncated singular value decomposition, where $m < \min \{M, N\}$ denotes the truncation level, $\mu_i, \tilde{u}_i, \tilde{v}_i, \ i=1,2,\cdots,m$ are the singular values and vectors of $\tilde{S}$. We usually need to provide a threshold $\gamma$ to choose $m$ and use the singular values only when $\mu_i \geq \gamma.$ Two popular methods are the soft thresholding \cite{DD} and hard thresholding \cite{GD2}.

In recent years, the advance of technology has lead to the observation of massive scale data, where the dimension of  the variable is comparable to the length of the observation. In this situation, the TSVD will lose its validity.  To address this problem, in the present paper, we consider the matrix denoising problem (\ref{defn_m11}) by assuming $M$ is comparable to $N$ and estimate $S_N$ in the following two regimes:\\

\noindent {\bf Regime (1)}.  $S_N$ is of low rank and we have prior information that its singular vectors are sparse; \\
 
\noindent {\bf Regime (2)}.  $S_N$ is of low rank and we have no prior information on the singular vectors. \\

In regime (1), $S_N$ is called simultaneously low rank and sparse matrix. This type of matrix  has been heavily used in biology. A typical example is from the study of gene expression data \cite{PGSR}. In \cite{YMB}, Yang, Ma and Buja also consider such problem but from a quite different perspective. They do not take the local behavior of singular values and vectors into consideration. Instead, they use an adaptive thresholding method to recover $S_N$ in (\ref{defn_m11}). In regime (2), we are interested in looking at what is the best we can do in this case.  A natural (and probably necessary) assumption is rotation invariance \cite{BABP},  as the only information we know about the singular vectors is orthonormality.  It is notable that, in this case, our result coincides with the results proposed by Gavish and Donoho \cite{GD1},  where they consider the estimator from another perspective and restrict the estimator to be conservative (see Definition 3 in \cite{GD1}).

In this paper, we will study the convergent limits and rates of the singular values and vectors for the sequence of matrices $\tilde{S}_N$ defined in (\ref{defn_m11}). 
For the rest of the paper, we will omit the subscript $N$ for convenience  and write 
 \begin{equation}\label{defn_m1}
 \tilde{S}=S+X.
 \end{equation}
 To avoid repetition, we summarize the technical assumptions of the noise matrix $X$.
\begin{assm}\label{assum_whitenoise} We assume $X$ is a white noise matrix, where the entries $x_{ij}$ of $X$ are i.i.d random variables such that
\begin{equation*}
\mathbb{E}x_{ij}=0, \ \mathbb{E}|x_{ij}|^2=\frac{1}{N}.
\end{equation*}
Furthermore, we assume that for $l \in \mathbb{N},$ there exists some constant $C_l>0,$ such that  
\begin{equation} \label{defn_xij}
\mathbb{E}|\sqrt{N}x_{ij}|^l \leq C_l. 
\end{equation}
\end{assm} 
Denote the SVD of $S$ as
\begin{equation}\label{eq_svds}
S=UDV^*=\sum_{k=1}^r d_i u_i v_i^*,
\end{equation}
 where $D=\text{diag} \{d_1, \cdots, d_r \}, \ U=(u_1, \cdots, u_r), \ V=(v_1, \cdots, v_r), $
and where $u_i \in \mathbb{R}^M, \ v_i \in \mathbb{R}^N$ are orthonormal vectors and $r$ is a fixed constant. We also assume $d_1>d_2 > \cdots >d_r>0.$
Then (\ref{defn_m1}) can be written as
\begin{equation} \label{defn_m2}
\tilde{S}=X+UDV^{*}.
\end{equation}
Throughout the  paper,  we are interested in the following setup
\begin{equation} \label{notation_d}
 c_N:=\frac{N}{M},  \  \ \lim_{N \rightarrow \infty} c_N = c \in (0,\infty).
 \end{equation}
It is well-known that for the noise matrix $X,$ the spectrum of $XX^*$ satisfies the celebrated  Marchenco-Pastur (MP) law \cite{MP} and the largest eigenvalue satisfies the Tracy-Widom (TW) distribution \cite{TW}. Specifically,  denote $\lambda_i:=\lambda_i(XX^*), i=1,2,\cdots, K$, where $K=\min\{M,N\},$ as the eigenvalues of $XX^*$ in a decreasing fashion,  we have that
\begin{equation} \label{notation_edges}
\lambda_1=\lambda_{+}+O(N^{-2/3}), \ \ \lambda_{+}=(1+c^{-1/2})^2,
\end{equation}
holds with high probability. Furthermore, denote $\xi_i, \zeta_i$ as the singular vectors of $X,$ for some large constant $C>0,$ with high probability, we have \cite{DXC}
\begin{equation*}
\max_{k} \{|\xi_i(k)|^2+|\zeta_i(k)|^2\}=O(N^{-1}), \  i \leq C.
\end{equation*}

To sketch the behavior of $\tilde{S},$ we consider the case when $r=1$ in (\ref{defn_m2}). Assuming that the distribution of the entries of $X$ is  bi-unitarily invariant, Benaych-Georges and Nadakuditi established the convergent limits in \cite{BGN} using free probability theory. Denote $\mu_i:=\mu_i(\tilde{S} \tilde{S}^*), i=1,2, \cdots, K, $ they proved that when $d>c^{-1/4},$ $\mu_1$ would detach from the spectrum of the MP law and become an outlier. And when $d<c^{-1/4},$ $\mu_1$  converges to $\lambda_+$ and sticks to the spectrum of the MP law. For the singular vectors, denote $\tilde{u}_i, \ \tilde{v}_i$ as the left and right singular vectors of $\tilde{S}, i=1, 2,\cdots, K.$  They proved that when $d>c^{-1/4},$ $\tilde{u}_1, \ \tilde{v}_1$ would be concentrated on cones with axis parallel to $u_1, \ v_1$ respectively, and the apertures of the cones converged to some  deterministic limits. And when $d<c^{-1/4}, $  $\tilde{u}_1, \ \tilde{v}_1$ will be asymptotically perpendicular to $u_1, \ v_1$ respectively. We point out that similar results have been proved for the Wigner matrices with additive deformation and covaraince matrices with multiplication perturbation. For such results, we refer the readers to \cite{BBP, BKYY, DXC3, KY2, KY, P, PRS1, PRS2}.

Our computation and proof rely on the isotropic local MP law \cite{BEKYY, KY1, PY}.  These results say that the eigenvalue distribution of the sample covariance matrix $XX^*$ is close to the MP law, down to the spectral scale containing slightly more than one eigenvalue. These local laws are formulated using the Green functions,
\begin{equation}\label{def_green}
\mathcal G_1(z):=(XX^{*}-z)^{-1} , \ \ \ \mathcal G_2 (z):=(X^{*} X-z)^{-1} , \ \ \ z=E+i\eta \in \mathbb{C}^{+}.
\end{equation}
%The above local MP laws have many applications in the local analysis of sample covariance matrices. To list a few, the rigidity of the eigenvalues \cite{KY1}, the completely delocalization of singular vectors \cite{DXC}, the edge and bulk universality of sample covariance matrices \cite{BEKYY, DY,KY1, KY2} and the outliers of general covariance matrices \cite{DXC3}.
To illustrate our results and ideas, we give an overview of the present paper.  As we have seen from \cite{DXC, DY}, the self-adjoint linearization technique is quite useful in dealing with rectangular matrices. Hence, in a first step, we denote by
\begin{equation} \label{linearconstruction1}
\tilde{H}=
  \begin{bmatrix}
    0 & z^{1/2} \tilde{S}\\
    z^{1/2}\tilde{S}^* & 0 
  \end{bmatrix}
  = \begin{bmatrix}
    0 & z^{1/2} X\\
    z^{1/2}X^*& 0 
  \end{bmatrix}
  +\begin{bmatrix}
    0 & z^{1/2} UDV^*\\
    z^{1/2}VDU^*& 0 
  \end{bmatrix}
  =H+\mathbf{U}\mathbf{D} \mathbf{U}^{*},
\end{equation} 
where  $\mathbf{D}, \mathbf{U}$ are defined as 
\begin{equation}\label{defn_mathbfuv}
\mathbf{D}:=  \begin{bmatrix}
    0 & z^{1/2}D \\
    z^{1/2}D & 0 
  \end{bmatrix}, \ 
  \mathbf{U}:=  \begin{bmatrix}
    U & 0\\
    0 & V 
  \end{bmatrix}.
\end{equation}
%(\ref{linearconstruction1}) is a very convenient expression. On one hand,  the eigenvalues of $\tilde{S} \tilde{S}^{*}$ can be uniquely characterized by the eigenvalues of $\tilde{H}$.  On the other hand, the Green functions of  $XX^*$ and $X^*X$ are contained in that of $H$ (see (\ref{green2})). 

Next we will give a heuristic description of our results.  We will always denote $\mu_1 \geq \cdots \geq \mu_K , \ K=\min\{M,N\}$
as the eigenvalues of $\tilde{S} \tilde{S}^*$ and  $\tilde{u}_i, \ \tilde{v}_i$ as the singular vectors of $\tilde{S}.$ And we denote $G(z)$ as the Green function of $H$. Consider $r=1$ in (\ref{defn_m2}) and by a standard perturbation discussion (see Lemma \ref{lem_perb}), we find that $\mu_1$ satisfies the equation $\det(\mathbf{U}^*G(\mu_1)\mathbf{U}+\mathbf{D}^{-1})=0.$ Using the isotropic local law in \cite{KY1}, we find that (see Lemma \ref{lem_anisotropic}) $G$ has a deterministic limit $\Pi$ when $N$ is large enough. Heuristically, the convergent limit of $\mu_1$ is determined by the equation $\det(\mathbf{U}^*\Pi(z)\mathbf{U}+\mathbf{D}^{-1})=0.$ An elementary calculation shows that, when $d>c^{-1/4}$, $\mu_1 \rightarrow p(d),$ where $p(d)$ is defined in (\ref{defn_pd}).

When $d>c^{-1/4},$ the largest eigenvalue $\mu_1$ will detach from the bulk and become an outlier around its \emph{classical location} $p(d)$. We would expect this happens under a scale of $N^{-1/3}.$ This can be understood in the following ways: increasing $d$ beyond the critical value $c^{-1/4}$, we expect $\mu_1$ to become an outlier,  where its location $p(d)$ is located at a distance greater than $O(N^{-2/3})$ from $\lambda_{+}.$ By using mean value theorem, the phase transition will take place on the scale when 
\begin{equation} \label{defn_bbphappens}
|d-c^{-1/4}| \geq O(N^{-1/3}).
\end{equation} 
When (\ref{defn_bbphappens}) happens, we also prove that 
\begin{equation}\label{intro_outlierev}
\mu_1=p(d)+O\left(N^{-1/2}(d-c^{-1/4})^{1/2}\right).
\end{equation}
Below this scale, we would expect the spectrum of $\tilde{S}\tilde{S}^*$ to stick to that of  $XX^{*}$. Especially, the largest eigenvalue $\mu_1$ still has the Tracy-Widom distribution with the scale $N^{-2/3}$, which reads as 
\begin{equation}\label{intro_nonoutlierev}
\mu_1=\lambda_+ +O(N^{-2/3}). 
\end{equation} 

For the singular vectors, when $d>c^{-1/4},$ we have $<u_1, \tilde{u}_1>^2 \rightarrow a_1(d),  <v_1, \tilde{v}_1>^2  \rightarrow a_2(d), $ where $a_1(d), a_2(d)$ are deterministic functions of $d$ and defined in (\ref{defn_a1a2}). For the local behavior, we will use an integral representation of Greens functions (see (\ref{intergralrepresentationv})). 
%However, when $r>1,$ if $d_i \approx d_j, \ i \neq j,$ we would expect that $\tilde{u}_i (\tilde{v}_i),  \tilde{u}_j(\tilde{v}_j)$ lie in the same eigenspace. And then we cannot distinguish the singular vectors. Therefore, in this paper, we assume that for $i \neq j,$ there exists some $\epsilon_0>0$, such that $d_i, d_j$ satisfy the following condition
%\begin{equation} \label{defn2_nonoverlapping}
%|p(d_i)-p(d_j)| \geq N^{-1/2+\epsilon_0}(d_i-c^{-1/4})^{1/2}.
%\end{equation}
%(\ref{defn2_nonoverlapping}) is referred as non-overlapping condition in \cite{BKYY, KY2}, it ensures that the eigenspace corresponding to different $d_i, \ i=1,\cdots, r$ can be well separated. This can be understood in the following ways:  when $d_i, \ d_j >c^{-1/4},$ the corresponding eigenvalues $\mu_i, \ \mu_j$ of $\tilde{S} \tilde{S}^*$ will converge to $p(d_i)$ and $p(d_j)$ respectively. Hence, (\ref{defn2_nonoverlapping}) ensures that the singular vectors can be distinguished individually.  
 Under the assumption that $d_i$'s are well-separated and satisfy (\ref{defn_bbphappens}), we prove that
\begin{equation} \label{intro_outlierevec}
<u_1, \tilde{u}_1>^2=a_1(d)+O(N^{-1/2}), \  <v_1, \tilde{v}_1>^2=a_2(d)+O(N^{-1/2}).
\end{equation} 
Below the scale of (\ref{defn_bbphappens}), we prove that
\begin{equation}\label{intro_nonoutlierevec}
<u_1, \tilde{u}_1>^2=O(N^{-1}), \  <v_1, \tilde{v}_1>^2=O(N^{-1}). 
\end{equation}

Armed with (\ref{intro_outlierev}), (\ref{intro_nonoutlierev}), (\ref{intro_outlierevec}) and (\ref{intro_nonoutlierevec}), we can go to the matrix denoising problem (\ref{defn_m2}) under the two different regimes. In the first regime, we assume there exists sparse structure of the singular vectors, in the case when $d>c^{-1/4},$ we would expect $\tilde{u}_1, \ \tilde{v}_1$ to be sparse as well. Hence, $\tilde{S}$ will be of sparse structure.  
%Figure \ref{figure1} is an example of such matrix.
%\begin{figure}[htb]
%\centering
%\includegraphics[height=8cm,width=11cm]{matrix.eps}
%\caption{ $S=duv^*$ is a $300 \times 600$ matrix, where $u, \ v$ are sparse vectors and $X$ is a random Gaussian matrix in (\ref{defn_m1}).}
%\label{figure1}
%\end{figure}
Therefore, by suitably choosing a submatrix of $\tilde{S}$ and doing SVD for the submatrix, we can get an estimator for the singular vectors.  Our novelty is to truncate singular values and vectors simultaneously. For the estimation of singular values, we can reverse (\ref{intro_outlierev}) to get the estimator for $d.$ For the singular vectors, based on (\ref{intro_nonoutlierevec}), the truncation level should be much larger than $N^{-1/2}$ and we will use K-means clustering algorithm to choose such level. However, when $d<c^{-1/4},$ we can estimate nothing according to (\ref{intro_nonoutlierev}) and (\ref{intro_nonoutlierevec}).

In the second regime, as we have no prior information whatsoever on the true eigenbasis of $S,$ the only possibility is to use the eigenbasis of $\tilde{S}.$ This is equivalent to the assumption of rotation invariance. We will propose a consistent rotation invariant estimator (RIE) $\Xi(\tilde{S}),$ which  satisfies the following condition, 
\begin{equation} \label{defn_rotationinvarint}
\Omega_1 \Xi(\tilde{S}) \Omega_2= \Xi(\Omega_1 \tilde{S} \Omega_2),
\end{equation}
where $\Omega_1, \Omega_2$ are orthogonal (rotation) matrix in $\mathbb{R}^M, \mathbb{R}^N$ respectively.  Before concluding this section, we list our main contributions of this paper: \\

\noindent (i). We systematically study the local behavior of the singular values and vectors for  finite rank perturbation of large dimensional rectangular matrices of model (\ref{defn_m2}). We compute the convergent limits and rates for them.  \\

\noindent (ii). We provide two efficient estimators for the matrix denoising model (\ref{defn_m2}) under two different regimes.  We provide practical algorithms to compute the estimators. For the sparse estimation, as far as we know, our paper is the first one to truncate the singular values and vectors simultaneously.  \\

%\begin{rem} In this paper, the entries of $X$ are assumed to be independent. However, our results can be extended to a more general framework by replacing $X$ by $\Sigma^{1/2}X,$ where $\Sigma$ is a positive definite matrix satisfying some regularity condition. The details of the discussion can be found in \cite{DXC3}.
%\end{rem}

This paper is organized as follows. In Section \ref{section_results}, we give the main results of this paper. In Section \ref{section_statistics}, we propose the estimators for (\ref{defn_m2}) under two regimes.   In Section \ref{section_tool}, we record the basic tools for the proof of the main theorems. In Section \ref{section_ev}, we prove the main theorems listed in Section \ref{section_results}. 
%In the supplementary material \cite{DBS}, we provide some extra technical proofs.

\noindent{\bf Conventions.}  For two quantities $a_N$ and $b_N$ depending on $N$, the notation $a_N = O(b_N)$ means that $|a_N| \le C|b_N|$ for some positive constant $C>0$, and $a_N=o(b_N)$ means that $|a_N| \le c_N |b_N|$ for some positive constants $c_N\to 0$ as $N\to \infty$. We also use the notation $a_N \sim b_N$ if $a_N = O(b_N)$ and $b_N = O(a_N)$. We define the minimum of any two reals $a, b$ by $ a\wedge b.$  For  any matrix $A$, we denote by $A^{*}$ as the transpose of $A$ and  $||A||_F$ the Frobenius norm of $A$. We will also use $\bm{\sigma}(H)$ to denote the spectrum for any square matrix $H.$ And for any  rectangular matrix $S$ we use $\sigma_i(S)$ to denote its $i$-th largest singular value.

\section{Main results}\label{section_results}

Throughout the paper, we always use $\epsilon_1$ for a small constant and $D_1$ for a large constant.  Denote $\mathcal{R}:=\{1,2,\cdots, r\}$  and $\mathcal{O}$ as a subset of of $\mathcal{R}$ by
\begin{equation} \label{defn_mathcalo}
\mathcal{O}:=\{i: d_i \geq c^{-1/4}+N^{-1/3+\epsilon_0}\}, \ \epsilon_0 > \epsilon_1 \ \text{is a small constant},
\end{equation}
and  the number of outlier singular values as
\begin{equation}\label{defn_k+}
k^+=|\mathcal{O}|.
\end{equation}
Our results can be extended to a more general domain by denoting  $\mathcal{O}^{\prime}:=\{i: d_i\geq c^{-1/4}+N^{-1/3}\}.$
We will not pursue this generalization. For more details, we refer to \cite{BKYY}. 
\noindent For any subset $A \subset \mathcal{O},$ we define the projections on the left and right singular subspace of $\tilde{S}$ by
\begin{equation}\label{defn_projection}
\mathbf{P}_{l}:=\sum_{i \in A} \tilde{u}_i \tilde{u}^*_i, \  \mathbf{P}_{r}:=\sum_{j \in A} \tilde{v}_j \tilde{v}^*_j. 
\end{equation} 
We also need the non-overlapping condition, which was firstly introduced in \cite{BKYY}. 
\begin{defn} \label{defn_nonoverlapping} For $i=1,2,\cdots,M,$ the non-overlapping condition is written as
\begin{equation}\label{defn1_nonoverlapping}
\nu_i(A) \geq (d_i-c^{-1/4})^{-1/2}N^{-1/2+\epsilon_0},  
\end{equation}
where $\epsilon_0$ is defined in  (\ref{defn_mathcalo}) and $\nu_i(A)$  is defined by 
\begin{equation} \label{defn_nuai}
 \nu_i(A):= 
\begin{cases}
\min_{j \notin A}|d_i-d_j|, \ \text{if} \ i \in A, \\
\min_{j \in A}|d_i-d_j|, \ \text{if} \ i \notin A. 
\end{cases}
\end{equation}
\end{defn}

With the above preparation,  we state our main results of the singular values of $\tilde{S}.$ Denote 
\begin{equation} \label{defn_pd}
p(d)=\frac{(d^2+1)(d^2+c^{-1})}{d^2}.
\end{equation}
Recall $\tilde{S}$ defined in (\ref{defn_m2}) and $\mu_i$ are the eigenvalues of $\tilde{S} \tilde{S}^*.$ 
\begin{thm}\label{thm_location}
Under Assumption \ref{assum_whitenoise} and the assumption of (\ref{notation_d}),  for $ i=1,2,\cdots, k^+,$ where $k^+$ is defined in (\ref{defn_k+}), there exists some large constant $C>1$ such that $C\epsilon_1 < \epsilon_0,$ when $N$ is large enough, with $1-N^{-D_1}$ probability, we have
\begin{equation} \label{thm_location_outlier}
|\mu_i-p(d_i)| \leq N^{-1/2+C\epsilon_0}(d_i-c^{-1/4})^{1/2},
\end{equation}
where $p(d_i)$ is defined in (\ref{defn_pd}).  Moreover, for $j=k^++1, \cdots, r,$ we have 
\begin{equation}  \label{thm_location_bulk}
|\mu_j-\lambda_+| \leq N^{-2/3+C\epsilon_0},
\end{equation}
where $\lambda_+$ is defined in (\ref{notation_edges}).
\end{thm}

The above theorem gives precise location of the outlier singular values and the extremal non-outlier singular values. For the outliers, they locate around their  classical locations $p(d_i)$ and for the non-outliers, they locate around $\lambda_+.$   The results of the singular vectors are given by the following theorem.  Denote 
\begin{equation}\label{defn_a1a2}
 a_1(d)=\frac{d^4-c^{-1}}{d^2(d^2+c^{-1})}, \ a_2(d)=\frac{d^4-c^{-1}}{d^2(d^2+1)}.
\end{equation}
\begin{thm} \label{thm_sigularvector}Under Assumption \ref{assum_whitenoise} and the assumptions of (\ref{notation_d}) and (\ref{defn1_nonoverlapping}), for all $i, j=1,2,\cdots, r,$ there exists some constant $C>0,$ with $1-N^{-D_1}$ probability, when $N$ is large enough, we have 
\begin{equation}\label{thm_left}
\left |<u_i, \mathbf{P}_{l}u_j>-\delta_{ij} \mathbf{1}(i \in A) a_1(d_i)\right | \leq N^{\epsilon_1}R(i,j,A,N), 
\end{equation}
\begin{equation}\label{thm_right}
|<v_i, \mathbf{P}_{r}v_j>-\delta_{ij} \mathbf{1}(i \in A) a_2(d_i)|\leq N^{\epsilon_1}R(i,j,A,N),
\end{equation}
where $a_1(d), a_2(d)$ are defined in (\ref{defn_a1a2}) and $R(i,j,A,N)$ is  defined as
\begin{align*}
R(i,j,A,N):= N^{-1/2}\left[\frac{\mathbf{1}(i \in A, j \in A) }{(d_i-c^{-1/4})^{1/2}+(d_j-c^{-1/4})^{1/2}} + \mathbf{1}(i \in A, j \notin A)\frac{(d_i-c^{-1/4})^{1/2}}{|d_i-d_j|}  \nonumber \right. \\
\left. + \mathbf{1}(i \notin A, j \in A)\frac{(d_j-c^{-1/4})^{1/2}}{|d_i-d_j|} \right]+N^{-1}\left[(\frac{1}{\nu_i}+\frac{\mathbf{1}(i \in A)}{|d_i-c^{-1/4}|})(\frac{1}{\nu_j}+\frac{\mathbf{1}(j \in A)}{|d_j-c^{-1/4}|})\right].
\end{align*}
Moreover, fix a small constant $\tau>0,$ for $ k^++1\leq j \leq (1-\tau)K,$ denote $\kappa^d_j:= N^{-2/3}(j \wedge (K+1-j))^{2/3}$, we have
\begin{equation}\label{thm_non_left}
|<u_i,  \tilde{u}_j>^2|\leq \frac{N^{C\epsilon_0}}{N((d_i-c^{-1/4})^2+\kappa^d_j)}, \ i=1,2, \cdots, r,
\end{equation}
and 
\begin{equation}\label{thm_non_right}
| <v_i, \tilde{v}_j>|^2\leq \frac{N^{C\epsilon_0}}{N((d_i-c^{-1/4})^2+\kappa^d_j)}, \ i=1,2,\cdots, r.
\end{equation}
Furthermore, if $c \neq 1,$ (\ref{thm_non_left}) and (\ref{thm_non_right}) hold for all $j=k^++1, \cdots, M. $
\end{thm}
\begin{rem} The assumption $ j \leq (1-\tau)K$ ensures that $\mu_j \geq \delta,$ for some constant $\delta>0$.  When $c \neq 1,$ it is guaranteed as we will see  from Lemma \ref{lem_rigidity} that $\mu_j \geq (1-c^{-1/2})^2/2. $ We need $\mu_j \geq \delta$ for the technical purpose of the application of the local laws. 
\end{rem}

Next we will give some examples to illustrate our results. We assume that $c \neq 1.$
\begin{exam}
 (1). Consider the right singular vectors and let $A=\{i\}$, we have
\begin{equation*}
|<v_i, \tilde{v}_i>^2-a_2(d_i)| \leq N^{\epsilon_1}\left[ \frac{1}{N^{1/2} (d_i-c^{-1/4})^{1/2}}+\frac{1}{N\nu^2_i(d_i-c^{-1/4})^2}\right].
\end{equation*}
This implies that, the cone concentration of the singular vector holds if $i \in \mathcal{O}$ and the non-overlapping condition (\ref{defn1_nonoverlapping}) holds. Furthermore,  if $d_i$ is well-separated from both  the critical point $c^{-1/4}$ and  the other outliers, the error bound is of order $\frac{1}{\sqrt{N}}.$ \\
(2).  Let $A=\{i\}$ and for  $1 \leq j \neq i \leq r$,  we have 
\begin{equation*}
|<v_j, \tilde{v}_i>^2| \leq \frac{N^{\epsilon_1}}{N(d_i-d_j)^2}.
\end{equation*}
Hence, if $|d_i-d_j|=O(1),$ then $\tilde{v}_i$ will be completely delocalized in any direction orthogonal to $v_i.$ \\
\noindent (3). If $i \in \mathcal{O},\ j \notin \mathcal{O},$ then we have
\begin{equation*}
|<v_i,\tilde{u}_j>^2| \leq \frac{N^{C\epsilon_0}}{N((d_i-c^{-1/4})^2+\kappa^d_j)}.
\end{equation*}
Hence, when $|d_i-c^{-1/4}|=O(1)$ or $\kappa^d_j=O(1),$ $\tilde{u}_j$ will be completely delocalized in the direction of $v_i.$ The first case reads as $\mu_i$ is an outlier and the second case as that $\mu_j$ is in the bulk of the spectrum of $\tilde{S}\tilde{S}^*.$
\end{exam}
 Before concluding this section, we use the following figure to illustrate the accuracy of the proposed bounds in (\ref{thm_location_outlier}), (\ref{thm_left}) and (\ref{thm_right}). We consider the rank one perturbation $\tilde{S}=duv^* +X,$ where $X$ is a Gaussian random matrix with mean zero and variance $1/N$ and $u,v$ are sparse vectors generated from the $R$ package $\mathit{R1magic}$. 

To avoid the influence of the constant, we consider the ratio between the empirical bound and dominated part, i.e.,  for $d>c^{-1/4},$ we will consider 
\begin{equation*}
R_1=\Phi_1 |\mu_1-p(d)|, \ R_2=\Phi_2 |\langle u, \tilde{u}_1 \rangle^2 -a_1(d) |, \  R_3=\Phi_2|\langle v, \tilde{v}_1 \rangle^2 -a_2(d) |,  
\end{equation*}
where $\Phi_1:=\sqrt{N} (d-c^{-1/4})^{-1/2}$ and $\Phi_2:=\sqrt{N (d-c^{-1/4})}.$ We consider the cases $c=0.5$ and $c=2, $ and choose $d=2.$ 
For each $N,$ we record the averaged ratios for $R_i,i=1,2,3,$ using 1,000 repetitions and plot these ratios for a variety of choices (in total 181) of $N$ between $200$ and $2000.$ We can conclude from Figure \ref{bound1} that these ratios are around some fixed constants independent of $N.$  

\begin{figure}[htp]
\begin{minipage}{.45\textwidth}
  \includegraphics[width=1.05\linewidth]{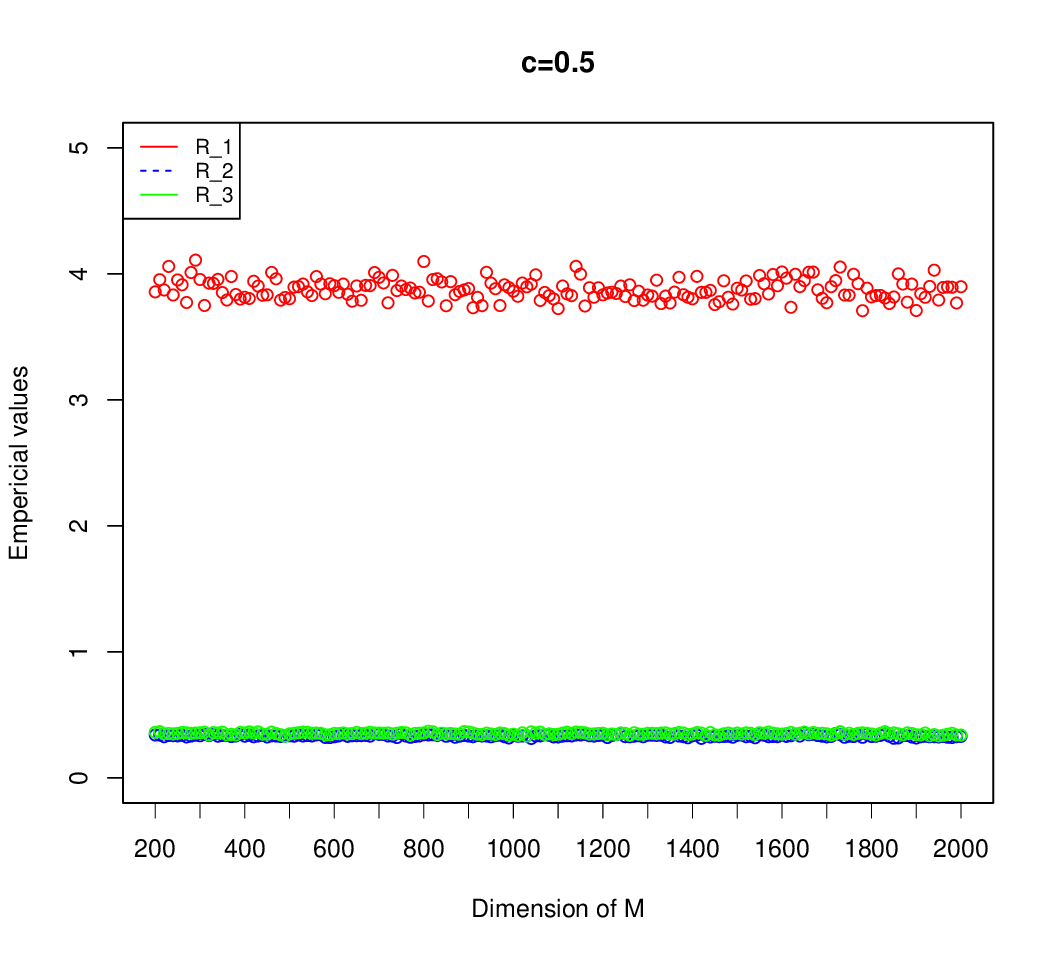} 
\end{minipage}%
\begin{minipage}{.45\textwidth}
  \includegraphics[width=1.05\linewidth]{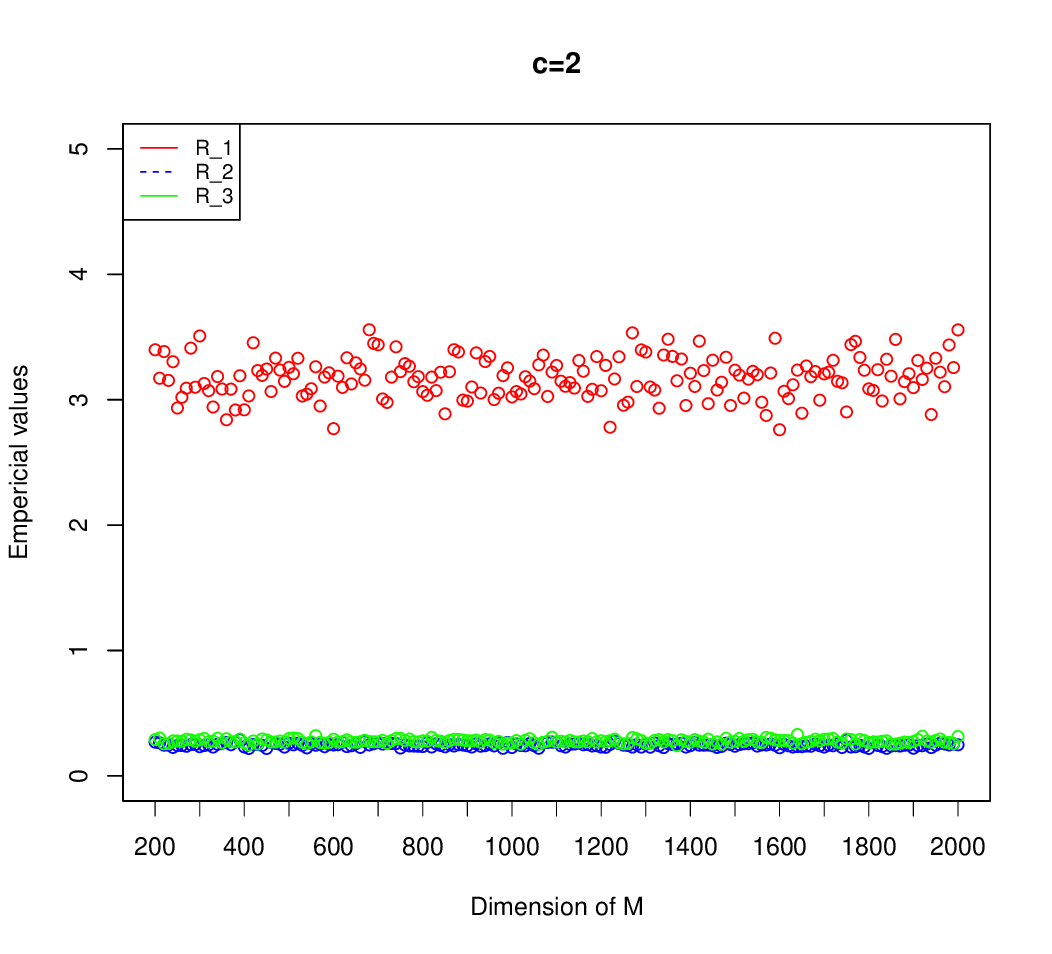} 
\end{minipage}
\caption{ We can see from the above figure that $R_1, R_2, R_3$ are independent of $N$. Further,  the left and right singular vectors have the same bounds.}
 \label{bound1}
\end{figure}

\section{Statistical applications }\label{section_statistics}

\subsection{Sparse estimation}  In the present application, we study the denoising problem (\ref{defn_m1}), where $S$ is sparse in the sense that the nonzero entries are assumed to be confined on a block. We assume that $u_i, \ v_i$ are sparse and introduce the following definition to precisely describe the sparsity.
\begin{defn} \label{defn_sparse}  For any vector $\nu \in \mathbb{R}^N$, $\nu$ is a sparse vector  if  there exists a subset $\mathbb{N}^{*} \subset \{1,2,\cdots, N \}$ with $|N^*|=O(1),$ such that 
\begin{equation*}
|\nu(i)|=
\begin{cases}
O(1), &  i \in \mathbb{N}^* ;\\
O(N^{-1/2}), & \text{otherwise}.
\end{cases}
\end{equation*} 
\end{defn}
Next we will propose an estimator for $S$ by estimating the singular values and vectors separately.  As can been see from Theorem \ref{thm_location}, we can estimate the true outlier singular values from their corresponding sample values.  To ease our discussion,  we impose the following stronger assumptions on the outlier singular values of $S.$
\begin{assm}\label{assum_statsingular}  For  $i,j=1,2, \cdots, k^{+},$ we assume that there exists some constant $\delta>0,$ such that $$d_i>c^{-1/4}+\delta, \ |d_i-d_j| \geq \delta, \ i \neq j. $$
\end{assm}
Note that the above assumption is a stronger version of (\ref{defn_mathcalo}) and widely used in the practical applications \cite{GD1, RRN2014, NE, PYao}. We first estimate the number $k^+$ of outlier singular values. In \cite{NE}, $k^+$ is referred as the \textit{effective number of identifiable signals} and the author provided an information theoretic estimator by minimizing the Akaike Information Criterion (AIC). Furthermore, some other
useful statistics have been proposed to effectively estimate the number of spikes in the spiked covariance matrix model, for instance the differences between consecutive eigenvalues in \cite{PYao}. By Theorem \ref{thm_location}, when $i \leq k^+,$ we expect $\mu_i/\mu_{i+1}$ will be away from one and when $i>k^+,$ it will be close to one.  In the present paper, we will employ the ratios of consecutive sample singular values \cite{LY} as our statistic. For $\tau:=O(N^{-\alpha})$ satisfying
\begin{equation}\label{eq_alphaconsi}
0<\alpha<\frac{2}{3}, 
\end{equation}
we denote (Recall $K=\min\{M,N\}$.)
\begin{equation} \label{defn_q}
q=\argmaxB_{i} \Big \{1 \leq i \leq K:  \mathcal{R}_i >1+\tau \Big \}, \ \tau>0, \ \mathcal{R}_i=\frac{\mu_i}{\mu_{i+1}}.
\end{equation}
We summarize the property of $q$ as the following proposition and its proof can be found in the supplementary material \cite{DBS}. 
\begin{prop}\label{prop_choiseoutlier} Under the assumptions of Theorem \ref{thm_location}  and Assumption \ref{assum_statsingular}, for some $\tau=O(N^{-\alpha})$ satisfying (\ref{eq_alphaconsi}), we have that
\begin{equation*}
\mathbb{P}(q=k^{+})=1-o(1).
\end{equation*}
\end{prop}

In practice, for the choice of $\tau$, we employ the automatic calibration procedure of \cite[Section 4]{PYao}. The idea is to use the ratio of the first two largest eigenvalues of a Wishart matrix, i.e., an $M \times N$ random Gaussian matrix satisfying Assumption \ref{assum_whitenoise}. Indeed, we need to search the eigenvalue index such that the ratio of two consecutive eigenvalues of $\tilde{S} \tilde{S}^*$ is much larger than $1+\tau$ corresponding to that of $XX^*$. In detail, we will use the following procedure to calibrate $\tau.$ 
\begin{enumerate}
\item[(1).] Generate a sequence (say 1,000) of   $M \times N$ random Gaussian matrices $Z_k, k=1,2,\cdots, 1000$ satisfying Assumption \ref{assum_whitenoise}. Calculate the ratios of the first and second eigenvalue of $Z_k Z_k^*$ and write them as $\mathcal{R}_{1,k},k=1,2,\cdots,1,000.$
\item[(2).] For a given large probability $\beta,$(say $\beta=0.98$ as suggested by \cite{PYao}), find the value $\tau$ such that
\begin{equation*}
\frac{\#\{k: \mathcal{R}_{1,k}-1 \leq \tau\}}{1000}=\beta.
\end{equation*}  
\end{enumerate} 
For $c=2, $ we find that $\tau=0.0577$ for $M=300$ and $\tau=0.0372$ for $M=500.$ These will be used later for our simulation studies.

%\noindent Denote
%\begin{equation} \label{defn_q}
%q=\argminB_{i} \{1 \leq i \leq K: \mu_i \leq \lambda_{+}+N^{-2/3+\tau}\}, \ \tau>0 \ \text{is a small constant}, 
%\end{equation}
%where $\lambda_+$ is defined in (\ref{notation_edges}). Therefore, $q$ is defined as the index of the first extremal non-outlier eigenvalue.   As $p(d)$ is an increasing function of $d$ and  $p(c^{-1/4})=\lambda_+,$ we conclude that there exist $q-1$ outliers and a phase transition will happen after $\mu_q.$ 

%Figure \ref{fig2} shows such a phenomenon.
%\begin{figure}[htb]
%\centering
%\includegraphics[height=8cm,width=14cm]{eigenvalue.eps}
%\caption{Eigenvalue phase transition. $X$ is a $100 \times 200 $  random Gaussian matrix. For $S$ defined in (\ref{defn_m2}),  $r=5,$ with $d_1=4, d_2=3, d_3=2.5, d_4=1.5, d_5=0.1.$ As $d_5< 2^{-1/4}<d_4<d_3<d_2<d_1,$ we would expect four outliers. Hence, $q=5$.}
%\label{fig2}
%\end{figure}

With the above notations, we provide the  \emph{stepwise SVD} {\bf Algorithm} \ref{CHalgorithm}  to recover $S$ in (\ref{defn_m1}). As $u_i, v_i$ are sparse,  we need to find a submatrix of $\tilde{S}$ by a suitable truncation. Instead of simply truncating the singular values \cite{GD1, YMB},  we truncate the singular values and vectors simultaneously.

\begin{algorithm}
\caption{Stepwise SVD}
\label{CHalgorithm}
\begin{algorithmic}[1]
\State Do SVD for $\tilde{S}=\sum_{i=1}^{K} \mu_i \tilde{u}_i \tilde{v}_i^*,$ and do the initialization $\tilde{S}_1=\tilde{S}=\sum t^1_i \tilde{u}_i^1 (\tilde{v}_i^1)^*.$
\While{$ 1 \leq
 j \leq q$ } 
\State $\hat{d}_j=p^{-1}((t^j_1)^2),$ where $p^{-1}(x)$ is the inverse of the function defined in (\ref{defn_pd}).
\State Use two thresholds $\alpha_{u_j} \gg \frac{1}{\sqrt{M}}, \ \alpha_{v_j} \gg \frac{1}{\sqrt{N}},$ and denote
\begin{equation}\label{defn_ij}
 I_j:=\{ 1 \leq k \leq M: |\tilde{u}^j_1(k)|\geq \alpha_{u_j} \}, \ J_j:=\{ 1 \leq k \leq N: |\tilde{v}^j_1(k)|\geq \alpha_{v_j} \}.
\end{equation}
\State Do SVD for the block matrix $\tilde{S}_{b}=\tilde{S}_j[I_j, J_j]=\sum \rho_i u^j_i (v^j_i)^*.$ 
\State Assume $I_j=\{k_1, \cdots, k_j\},$ construct
$\hat{u}_j$ by letting
\begin{equation*} 
\hat{\mu}_j(k_j)=
\begin{cases}
\mu_1^j(j), &  k_j \in I_j, \\
0, & \text{otherwise}.
\end{cases}
\end{equation*}
Similarly, we can construct $\hat{v}_j. $
\State Let $\tilde{S}_{j+1}=\tilde{S}_{j}-\hat{d}_j \hat{u}_j \hat{v}_j^*$ and do SVD for  $\tilde{S}_{j+1}=\sum t_i^{j+1} \tilde{u}^{j+1}_i(\tilde{v}^{j+1}_i)^*.$
\EndWhile
\State Denote  $\hat{S}=\sum_{k=1}^{q} \hat{d}_k \hat{u}_k \hat{v}_k^*$ as our estimator. 
\end{algorithmic}
\end{algorithm}

{\bf Algorithm} \ref{CHalgorithm} provides us a way to recover $S$ stepwisely. We first estimate $d_1,  u_1, v_1$ using the estimation $\hat{d}_1, \hat{u}_1, \hat{v}_1,$  then $d_2, u_2, v_2$ by analyzing $\tilde{S}-\hat{d}_1 \hat{u}_1 \hat{v}_1^*.$ In each step, we only need to look at the largest singular value and its associated singular vectors.  It is notable that, we drop all  the eigenvalues $\mu_i$ of $\tilde{S} \tilde{S}^*$ when $i<q$ and
\begin{equation}\label{sparseestimationshrinakge}
\hat{d}_i=\mathbf{1}(i \geq q)p^{-1}(\mu_i).
\end{equation}

% As we can see from (\ref{intro_outlierev}),  (\ref{sparseestimationshrinakge}) outperforms the commonly used hard thresholding by simply denoting \cite{GD2, YMB}
%\begin{equation*}
%\hat{d}_i=\mathbf{1}(\mu_i>\gamma) \mu_i, \ \gamma>0 \ \text{is a given threshold}.
%\end{equation*} 

Our methodology relies on truncating singular values and vectors simultaneously. As illustrated in (\ref{defn_ij}), the thresholds $\alpha_u$ and $\alpha_v$ play the key roles in recovering the sparse structure of the singular vectors. It will be proved in Section \ref{section_results} that any threshold satisfying (\ref{defn_ij}) should work when $N$ is sufficiently large. In the finite sample framework (when $N$ is not quite large), we  employ the K-means  algorithm \cite[Section 10.3.1]{JWHT} to stabilize  the recovery of  the sparse structure of $S$.  The reason behind is, the entries in the singular vectors $\tilde{u}_i, \tilde{v}_i$ can be well classified into two categories.  Denote the index sets $C_{u}^j, C_{v}^j$ getting from the K-means algorithm, where they satisfy 
\begin{equation} \label{boundskmeans}
\min_{k \in C_{u}^j} |\tilde{u}_1^j(k)| \gg \frac{1}{\sqrt{M}},  \  \min_{k \in C_{v}^j} |\tilde{v}_1^j(k)| \gg \frac{1}{\sqrt{N}}.
\end{equation}
We now replace (\ref{defn_ij}) with the following step:
\begin{itemize}
\item  Do K-means clustering to partition the entries of $\tilde{u}_1^j, \ \tilde{v}_1^j$ into two classes, where 
\begin{equation}\label{defn_ij2}
 I_j:=\{ 1 \leq k \leq M: k \in C^j_{u} \}, \ J_j:=\{ 1 \leq k \leq N: k \in C^j_{v} \},
\end{equation} 
where $C_{u}^j, C_{v}^j$ satisfy (\ref{boundskmeans}).
\end{itemize}

Next, we summarize the theoretical properties of {\bf Algorithm} \ref{CHalgorithm} as the following theorem and leave its proof into the supplementary material \cite{DBS}. 
\begin{thm} \label{thm_sparse}
With prior information that $u_i, v_i$ are sparse in the sense of Definition \ref{defn_sparse}, under the assumptions of Theorem \ref{thm_location} and \ref{thm_sigularvector},
and Assumption \ref{assum_statsingular}, there exists some $C>0,$  with $1-o(1)$ probability, for the estimator $\hat{S}$ getting from {\bf Algorithm \ref{CHalgorithm}}, we have 
\begin{equation*}
||\hat{S}-S ||_F \leq N^{-1/2+C\epsilon_0}+\sqrt{\sum_{i=k^++1}^r d_i^2}.
\end{equation*}
\end{thm}

Before concluding this subsection, we compare our method with other different algorithms.
In \cite{YMB}, the authors proposed another algorithm from a quite different perspective. They did not take the properties of the singular values and vectors of $\tilde{S}$ into consideration.  Instead, they used iterative thresholding on the rows of  $\tilde{S}$ to get an estimator. The algorithm is called \emph{sparse SVD}. Their algorithm can be regarded as the extension of TSVD on the submatrix of $\tilde{S}$. 

We use Table \ref{my-label} to compare the results of three algorithms, our stepwise SVD(SWSVD), the sparse SVD(SSVD) proposed by \cite{YMB} and the truncated SVD(TSVD).  For the implementation of  SSVD, we use the $\mathit{ssvd}$ package in R which is contributed by the first author of \cite{YMB}. From Table \ref{my-label}, we find that our method outperforms both the SSVD and TSVD in all the cases . Furthermore, the standard deviation is small, which implies that our estimation is quite stable.

\begin{table}[ht]
\begin{tabular}{llllllll}
\toprule[1.5pt]
      & M=300    &         &     &  & M=500   &         &   \\   
      & Sparsity & $L^2$ error norm & Std &  & Sparsity & $L^2$ error norm & Std \\ \cmidrule[1.5pt](lr){1-4}\cmidrule[1.5pt](l){5-8}
SWSVD & 0.05     & {\bf 0.043}       &  0.175     &  &    0.05      & {\bf 0.045}        & 0.189    \\
      & 0.1      & {\bf 0.614}         & 0.178    &  &  0.1 & {\bf 0.6}        & 0.16     \\
      & 0.2      & {\bf 0.822}         & 0.126    &  &  0.2        & {\bf 0.825}        & 0.137    \\
    & 0.45      &  {\bf 1.1}       &   0.114  &  & 0.45          &{\bf 1.09}         & 0.09         \\  \cmidrule[1.5pt](lr){1-4}\cmidrule[1.5pt](l){5-8}
SSVD  & 0.05     & 4.01       & 0.002   &  & 0.05          &   4.01      &  0.002    \\
      & 0.1      &  4.01       & 0.004     &  &  0.1        &  4.02       & 0.002     \\
       & 0.2      &  4.04       & 0.004     &  &  0.2        &  4.03        & 0.004    \\
     & 0.45     &    4.06    &  0.005     &  &0.45          & 4.08        &  0.004         \\
 \cmidrule[1.5pt](lr){1-4}\cmidrule[1.5pt](l){5-8}
TSVD  & 0.05     &  53.9    & 6.872   &  &     0.05     &    53.75    & 6.63     \\  
      & 0.1      & 53.72        & 6.63   &  & 0.1        &        53.38 & 6.71    \\
      & 0.2      & 52.33        & 7.01   &  & 0.2         &        52.2 & 6.65    \\
      & 0.45     & 51.043        &  2.49   &  &0.45          & 52.4         & 4.3  \\
\bottomrule[1.5pt]
\end{tabular}
\vspace{3.5pt}
\caption{Comparison of  the algorithms. We choose $r=2, c=2, d_1=7, d_2=4$ in (\ref{defn_m2}). The noise matrix $X$ is Gaussian. In the table, sparsity is defined as the ratio of non-zero entries and length of the vector and we assume that $u_i, v_i, i=1,2$ have the same sparsity. We highlight the smallest error norm. }
\label{my-label}
\end{table}

\subsection{ Rotation invariant estimation} \label{RIE} This subsection is devoted to recovering $S$ in (\ref{defn_m1}) assuming that no prior information about $S$ is available. In this regime, we will consider the rotation invariant estimator (RIE) satisfying (\ref{defn_rotationinvarint}). We conclude from \cite{BABP} that any RIE  shares the same singular vectors as $\tilde{S}.$ To construct the optimal estimator, we use the Frobenius norm as our loss function. Denote $ \hat{S}=\Xi (\tilde{S}), $ we  have  

%On the other hand, the optimality is defined with respect to a specific loss function.  

%The following definition characterizes the property of  (\ref{defn_rotationinvarint}), it is firstly introduced by Donoho and Gavish in \cite{GD1}.
%
%\begin{defn}[Rotation invariant loss function] A loss function $L_{M,N}(\cdot,\cdot)$ is rotation invariant if for any $A,B\in \mathbb{R}^{M \times N}$ and all $M,N$, we have $L_{M,N}(A,B)=L_{M,N}(UAV,UBV),$ for any orthogonal $U \in O_M, \ V \in O_N. $ 
%\end{defn}
%
%By definition, we find that any  loss functions associated with singular values are rotation invariant. We will consider the , 
\begin{equation}\label{defn_norm}
||S-\hat{S}||_F^2=\operatorname{Tr}(S-\hat{S})(S-\hat{S})^*.
\end{equation}
Therefore, the form of the  RIE can be written in the following way
\begin{equation} \label{defn_minimizeproblem}
\hat{S}=\argminB_{H \in \mathcal{M}(\tilde{U}, \tilde{V})} ||H-S||_F,
\end{equation}
where $\mathcal{M}(\tilde{U}, \tilde{V})$ is the class of $M \times N$ matrices whose left singular vectors are $\tilde{U}$ and right singular vectors are $\tilde{V}.$ Suppose $\hat{S}= \sum_{i=1}^K \eta_k \tilde{u}_k \tilde{v}_k^*, $
denote $\mu_{k_1 k}=<u_{k_1},\tilde{u}_k>, \  \nu_{k_1 k}=<v_{k_1}, \tilde{v}_k>,$ then by an elementary computation, we find
\begin{align} \label{rie_errordecomposition}
||S-\hat{S}||_F^2= & \sum_{k=1}^{r} (d_k^2+\eta_k^2) -2\sum_{k=1}^r d_k \eta_k \mu_{kk} \nu_{kk} \nonumber\\
+ & \sum_{k=r+1}^{K} \eta_k^2-2\sum_{k_1 \neq k_2}^r d_{k_1} \eta_{k_2} \mu_{k_1 k_2} \nu_{k_1 k_2}-2\sum_{k_1=r+1}^K \sum_{k_2=1}^{r} \eta_{k_1} d_{k_2} \mu_{k_2 k_1} \nu_{k_2 k_1}.
\end{align}
Therefore,  $\hat{S}$ is optimal if
\begin{equation} \label{defn_rie2}
\ \eta_k=<\tilde{u}_k, S\tilde{v}_k>=\sum_{k_1=1}^{r} d_{k_1} \mu_{k_1 k}\nu_{k_1 k},  \ k=1,\cdots, K.
\end{equation}
In the present paper, we use the following estimator for $\eta_k$ and will prove its consistency in Section \ref{section_results}. Recall (\ref{defn_q}), the estimator is denoted as
\begin{equation}\label{defn_rieestimator}
\hat{\eta}_k = 
\begin{cases}
\hat{d}_k a_1(\hat{d}_k) a_2(\hat{d}_k), &   k \leq q; \\
0, &  k > q.
\end{cases}, 
\end{equation}
where $\hat{d}_k=p^{-1}(\mu_k)$ and $a_1(x), a_2(x)$ are defined in (\ref{defn_a1a2}).
Denote 
\begin{equation} \label{eq_hatcals}
\hat{\mathcal{S}}=\sum_{k=1}^q  \hat{\eta}_k \tilde{u}_k \tilde{v}_k^*,
\end{equation}
It is notable that the convergent limits for the shrinkage $\hat{\eta}_k$ and MSE for $\hat{\mathcal{S}}$ have already been computed in \cite{RRN2014}. We next summarize the theoretical properties of our estimators as the following theorem.  Its proof can be found in the supplementary material \cite{DBS}. 
\begin{thm} \label{thm_rie}
(1).  Under the assumptions of Theorem \ref{thm_location} and \ref{thm_sigularvector},  there exists some large constant $C>0$ and small constant $\tau>0,$ with $1-o(1)$ probability,  we have $\hat{\eta}_k \rightarrow \eta_k, \ k=1,2,\cdots, K.$ Furthermore, for $ 1 \leq k \leq (1-\tau)K,$ we have
\begin{equation}\label{consiste_rie}
|\hat{\eta}_k-\eta_k| \leq \mathbf{1}(k \leq k^+)N^{-1/2+C\epsilon_0}+\mathbf{1}(k>k^+)N^{-1+C\epsilon_0}.
\end{equation}
Moreover, when $c\neq 1, $ (\ref{consiste_rie}) holds for all  $k=1,\cdots,K.$ 
(2). When $c \neq 1,$ there exists some constant $C>0,$ with $1-o(1)$ probability, for $\hat{\mathcal{S}}$ defined in (\ref{eq_hatcals}), we have 
\begin{equation*}
|| \hat{\mathcal{S}}-S ||_F^2 \leq \sum_{i=1}^r d_i^2-\sum_{i=1}^{k+} \Big(d_i a_1(d_i)a_2(d_i) \Big)^2 +N^{-1/2+C\epsilon_0}.
\end{equation*}
\end{thm}

Figure \ref{fig3} are two examples of the estimations of $\eta_k$. From the graph, we find that our estimator $\hat{\eta}_k$ is quite accurate. Figure \ref{fig4} records the relative improvement in average loss (RIAL) compared to TSVD, where the RIAL is defined as
\begin{equation} \label{defn_rial}
\text{RIAL}(N)=1-\frac{\mathbb{E}||\hat{\mathcal{S}}-S||_F}{\mathbb{E}||S_T-S||_F},
\end{equation}
and where $S_T$ is the TSVD estimation and $\hat{\mathcal{S}}$  the RIE. We conclude from the figure that our method provides better estimation compared to the TSVD. Similar results have been shown for the estimation of covariance matrices by Ledoit and P{\' e}ch{\' e} in \cite{LP}.

\begin{figure}[htp]
\begin{minipage}{.45\textwidth}
  \includegraphics[width=1.05\linewidth]{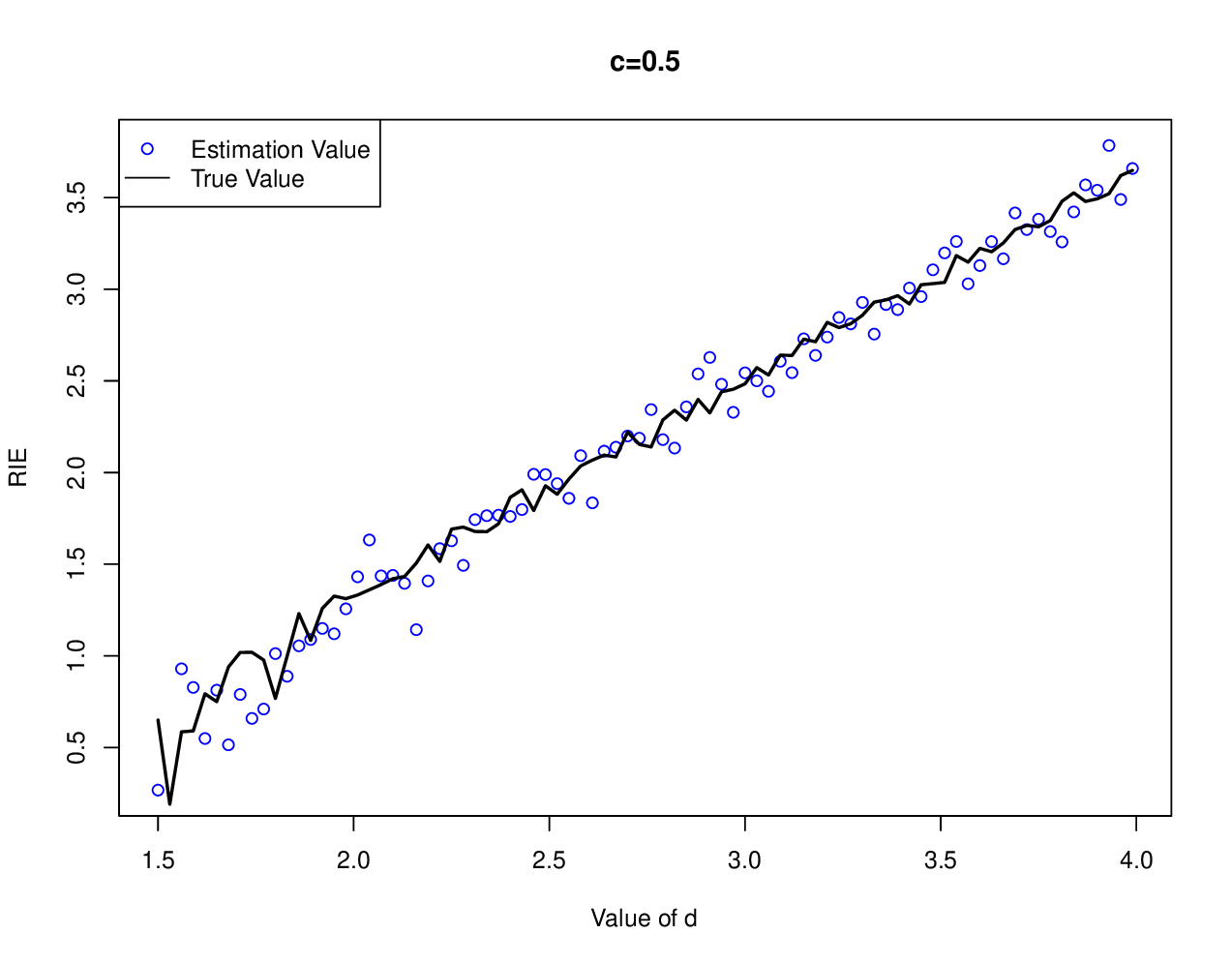} 
\end{minipage}%
\begin{minipage}{.45\textwidth}
  \includegraphics[width=1.05\linewidth]{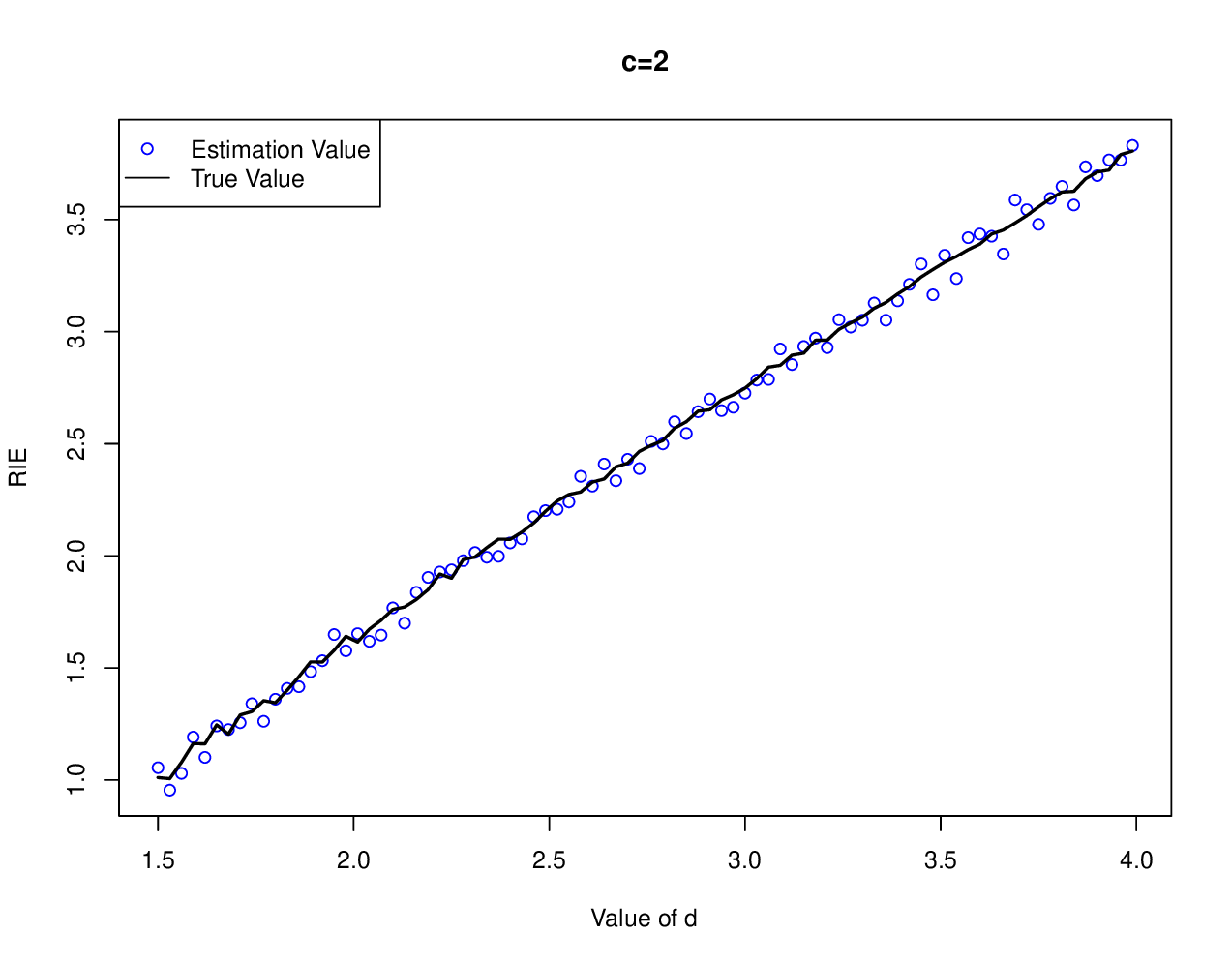} 
\end{minipage}
\caption{ RIE. We choose $r=1$ and $M=300$ for (\ref{defn_m2}). We estimate $\eta_1$ using the estimator (\ref{defn_rieestimator}) for  $c=0.5, 2$  with different values of $d$. The entries of $X$ are Gaussian random variables and the singular vectors satisfy the exponential distribution with rate 1.}
 \label{fig3}
\end{figure}
\begin{figure}[!htb]
\centering
\includegraphics[height=5cm,width=10cm]{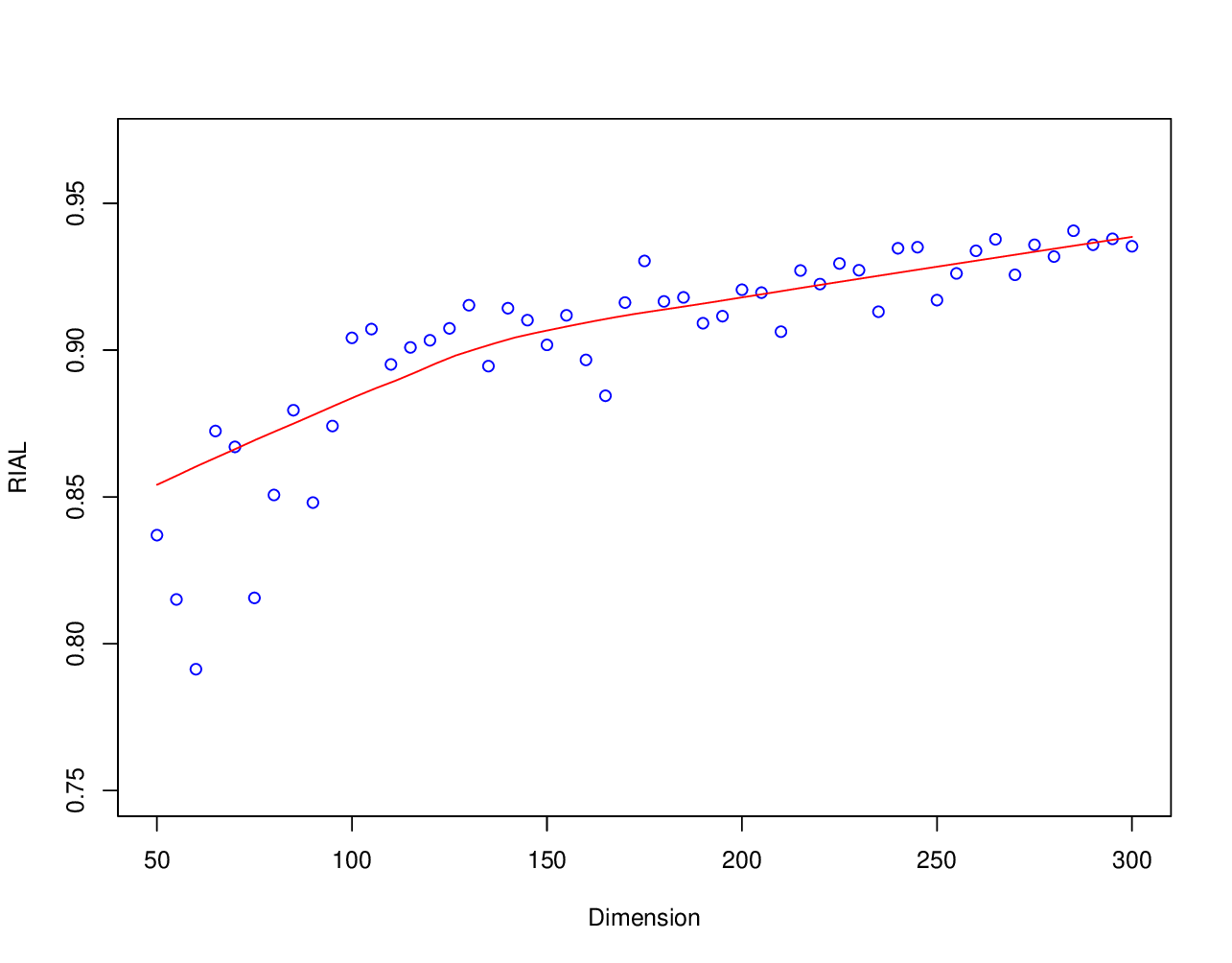}
\caption{ RIE compared to TSVD. We choose $r=1, d=4, c=2$ in (\ref{defn_m1}). X is a random Gaussian matrix and the entries of the singular vectors satisfy the exponential distribution with rate 1.  We perform 1000 Monte-Carlo simulations for each $M$ to simulate the RIAL defined in (\ref{defn_rial}). The red line indicates the increasing trend as $M$ increases. }
\label{fig4}
\end{figure}
\begin{rem} In \cite{GD1}, Donoho and Gavish get similar results from the perspective of optimal shrinkage. However, they need two more assumptions: (1). they drop the last two error terms in (\ref{rie_errordecomposition}) by assuming they are small enough (see Lemma 4 in their paper); (2). their estimators are assumed to be conservative, where  they assume the shrinker vanishes when the sample singular values are below $\lambda_+$  defined in (\ref{notation_edges}), i.e.,  for some constant $\gamma>0,$ 
\begin{equation*}
\eta_k=0, \ \text{when} \ \mu_k  \leq \lambda_++\gamma.
\end{equation*} 
However,  we find that the estimator defined in (\ref{defn_rieestimator}) can still be consistent even without these assumptions. 
\end{rem}

\section{Basic tools}\label{section_tool}

In this section, we introduce some notations and  tools which will be used in this paper. Recall that the empirical spectral distribution (ESD) of an $n \times n$ symmetric matrix $H$ is defined as
\begin{equation*}
F^{(n)}_H(\lambda):=\frac{1}{n} \sum_{i=1}^n \mathbf{1}_{\{\lambda_i(H) \leq \lambda\}}.
\end{equation*}
\noindent  We define the typical domain for $z=E+i\eta$ by
\begin{equation} 
\mathbf{D}(\tau)\equiv \mathbf{D}(\tau, N):= \{ z \in \mathbb{C}^{+}: \tau \leq E  \leq \tau^{-1}, \ N^{-1+\tau} \leq \eta \leq \tau^{-1}  \} \label{DOMAIN1},
\end{equation}
where $\tau>0$ is a small constant. Recall (\ref{notation_d}), we assume that $\tau<c_N <\tau^{-1}.$ 

\begin{defn} The Stieltjes transform of the ESD of $X^{*}X$ is given by
\begin{equation*}
m_2(z)\equiv m_2^{(N)}(z):=\int \frac{1}{x-z}dF^{(N)}_{X^*X}(x)=\frac{1}{N}\sum_{i=1}^N (\mathcal G_2)_{ii}(z)=\frac{1}{N} \mathrm{Tr} \, \mathcal G_2(z), \label{EEEE}
\end{equation*}
where $\mathcal{G}_2(z)$ is defined in (\ref{def_green}). Similarly, we can also define $m_1(z):= M^{-1}\mathrm{Tr} \, \mathcal G_1(z)$.
\end{defn}
Denote $m_{1c}(z):=\lim_{N \rightarrow \infty}m_1(z), \ m_{2c}(z):=\lim_{N \rightarrow \infty}m_2(z) $
be the Stieltjes transforms of limiting spectral distributions of $m_1(z), m_2(z)$. Using the identity $m_1 (z)= - \frac{1-c_N}{z}+c_N m_2 (z),$ we have 
\begin{equation} \label{relation_m1cm2c}
m_{1c}(z)=\frac{c-1}{z}+cm_{2c}(z).
\end{equation}

 \begin{defn} For $X$ satisfying (\ref{defn_xij}), under the assumption (\ref{notation_d}), the ESD of $XX^{*}$ converges weakly to the Marchenko-Pastur (MP) law as $N\to \infty$ \cite{MP}:
\begin{equation}\label{mp_density}
\rho_{1c}(x)dx=\frac{c}{2\pi}\frac{\sqrt{(\lambda_{+}-x)(x-\lambda_{-})}}{x}dx, \ \ \lambda_{\pm}=(1\pm c^{-\frac{1}{2}})^2.
\end{equation}
The Stieltjes transform of the MP law $m_{1c}(z)$ has the closed form expression (see (1.2) of \cite{JS})
\begin{equation}\label{defn_m1c}
m_{1c}(z)= \frac{1-c^{-1} -z + i\sqrt{(\lambda_{+}-z)(z-\lambda_{-})}}{2zc^{-1}}.
\end{equation}

\end{defn}
\begin{rem}
From (\ref{relation_m1cm2c}), we have that $m_2(z)$ converges to $m_{2c}(z)$ as $N\to \infty$, where
\begin{equation}
m_{2c}(z)=\frac{c^{-1}-1}{z}+c^{-1}m_{1c}(z) =\frac{ c^{-1} - 1 -z + i\sqrt{(\lambda_{+}-z)(z-\lambda_{-})}}{2z}. \label{GENERMP}
\end{equation}
It is notable that 
\begin{equation}\label{relationm1m2}
-z^{-1}(1+m_{2c}(z))^{-1}=m_{1c}(z).
\end{equation}
\end{rem}

Recall (\ref{linearconstruction1}) and $G(z)=(H-z)^{-1},$ by Schur's complement \cite{KY1}, it is easy to check that
\begin{equation} \label{green2}
G(z) = \left( {\begin{array}{*{20}c}
   { \mathcal{G}_1(z)} &  z^{-1/2}\mathcal{G}_1(z)X \\
   z^{-1/2}X^{*} \mathcal{G}_1(z) & \mathcal{G}_2(z) \\
\end{array}} \right),
\end{equation}
for $\mathcal G_{1,2}$ defined in (\ref{def_green}).  Denote the index sets
$\mathcal I_1:=\{1,...,M\}, \ \mathcal I_2:=\{M+1,...,M+N\}, \ \mathcal I:=\mathcal I_1\cup\mathcal I_2.$
Then we have
\begin{equation*}
m_1(z)=\frac{1}{M}\sum_{i\in \mathcal I_1}G_{ii}, \ \ m_2(z)=\frac{1}{N}\sum_{\mu \in \mathcal I_2}G_{\mu\mu}.
\end{equation*}
Similarly, we denote $\tilde{G}(z)=(\tilde{H}-z)^{-1},$ where $\tilde{H}$ is defined in (\ref{linearconstruction1}). Next we introduce the spectral decomposition of $\tilde{G}(z)$.  By (\ref{green2}), we have
\begin{equation}\label{main_representation}
\tilde{G}(z) = \sum_{k=1}^{K} \frac{1}{\mu_k-z} \left( {\begin{array}{*{20}c}
   { \tilde{u} \tilde{u}^*_k} &  z^{-1/2} \sqrt{\mu}_k \tilde{u}_k \tilde{v}_k^* \\
   z^{-1/2} \sqrt{\mu}_k\tilde{v}_k \tilde{u}_k^* & \tilde{v}_{k} \tilde{v}_{k}^*  \\
\end{array}} \right).
\end{equation}

As we have seen in (\ref{defn_pd}), the function $p(d)$ plays a key role in describing the convergent limits of the outlier singular values of $\tilde{S}.$  An elementary computation yields that $p(d)$ attains its global minimum when $d=c^{-1/4}$ and $p(c^{-1/4})=\lambda_+$, and
\begin{equation}\label{defn_pprime}
p^{\prime}(x) \sim (x-c^{-1/4}).
\end{equation}
To precisely locate the outlier singular values of $\tilde{S},$  we need to analyze
\begin{equation} \label{defn_ts}
 T^s(x):=\prod_{i=1}^s (xm_{1c}(x)m_{2c}(x)-d_i^{-2}). 
\end{equation}
%Figure \ref{fig6} is an example of $T^s(x).$
%\begin{figure}[!htb]
%\centering
%\includegraphics[height=8cm,width=16cm]{tx.eps}
%\caption{Graph for $T^4(x).$ We choose $c=2, s=4$ and the corresponding $d_i$ to be $3, 3.5, 4, 4.5$ receptively. }
%\label{fig6}
%\end{figure}
By (\ref{defn_m1c}) and (\ref{GENERMP}), when $x \geq \lambda_+,$ we have
\begin{equation} \label{notation_Tprimex}
xm_{1c}(x)m_{2c}(x)=\frac{x-(1+c^{-1})-\sqrt{(x+c^{-1}-1)^2-4c^{-1}x}}{2c^{-1}}.
\end{equation}
Next we collect the preliminary results of the properties of $T^s(x)$, whose proof will be provided in the supplementary  material \cite{DBS}.  
\begin{lem} \label{lem_zm1zm2z} Suppose $d_1 >d_2>\cdots>d_s>c^{-1/4}$, then we have that
there exist $s$ solutions of $T^s(x)=0$ and they are $p_i:=p(d_i),  i=1, 2, \cdots, s,$ write
\begin{equation} \label{tpequality}
T^s(p_i)=0.
\end{equation}
Furthermore, denote 
\begin{equation}\label{defn_caltz}
\mathcal{T}(x):=xm_{1c}(x)m_{2c}(x),
\end{equation}
$\mathcal{T}(x)$ is a strictly monotone decreasing function when $x>\lambda_{+}.$ 
\end{lem}

%
%\begin{rem} \label{remark_zmizm2z} By Lemma \ref{lem_zm1zm2z}, we find that, when $s$ is an even number, $T^s(x)$ is positive when $x<p(d_s)$ and $x>p(d_1)$ and alternating on the intervals $(p(d_{i+1}), p(d_i)), \ i=1,2,\cdots,s-1.$ And when $s$ is an odd number, $T^s(x)$ is positive when $x<d_s$ and negative when $x>d_1$ and alternating on the other intervals. 
%\end{rem}

 For $z \in \mathbf{D}(\tau)$ defined in (\ref{DOMAIN1}), denote 
\begin{equation} \label{defn_kappa}
\kappa:=|E-\lambda_{+}|.
\end{equation}
By (\ref{notation_Tprimex}), it is easy to check that
\begin{equation} \label{tzmc21}
\mathcal{T}(z)-c^{1/2}=\frac{z-\lambda_+-i \sqrt{(\lambda_+-z)(z-\lambda_{-})}}{2c^{-1}}.
\end{equation}
The following lemma summarizes the basic properties of $m_{2c}(z)$ and $\mathcal{T}(z)$, the estimates are based on the elementary calculations of (\ref{notation_Tprimex}) and (\ref{tzmc21}). Their proofs can be found in \cite[Lemma 3.3]{BEKYY} and \cite[Lemma 3.6]{BKYY}.
\begin{lem}\label{lem_boundsm1m2} 
 For any $z \in \mathbf{D}(\tau)$ defined in (\ref{DOMAIN1}), we have
\begin{equation*}
|\mathcal{T}(z)| \sim |m_{2c}(z)| \sim 1, \ | c^{1/2}-\mathcal{T}(z)| \sim |1-m^{2}_{2c}(z)| \sim \sqrt{\kappa+\eta},
\end{equation*}
and 
\begin{equation*}
\operatorname{Im} \mathcal{T}(z)  \sim
\operatorname{Im}m_{2c}(z) \sim
\begin{cases}
 \sqrt{\kappa+\eta}, & \text{if} \ E \in [\lambda_{-}, \lambda_{+}] ,\\
\frac{\eta}{\sqrt{\kappa+\eta}}, & \text{if} \ E \notin [\lambda_{-}, \lambda_{+}].
\end{cases},
\end{equation*}
as well as
\begin{equation} \label{lem_realboundt}
| \operatorname{Re} \mathcal{T}(z)-c^{1/2} | \sim
\begin{cases}
\frac{\eta}{\sqrt{\kappa+\eta}}+\kappa, \ E \in [\lambda_{-}, \lambda_{+}], \\
 \sqrt{\kappa+\eta}, \ \ \ \ E \notin [\lambda_{-}, \lambda_{+}].
 \end{cases}.
\end{equation}
\end{lem} 
The next lemma provides the local estimate on the derivative of $\mathcal{T}(x)$ on the real axis. We put its proof in the  supplementary material \cite{DBS}.
\begin{lem}  \label{lem_derivativebound}
For $d>c^{-1/4},$ denote
$I_{d}:=[x_{-}(d), x_{+}(d)], \ x_{\pm}(d):=p(d) \pm N^{-1/2+\epsilon_0}(d-c^{-1/4})^{1/2},$ where $\epsilon_0$ is defined in (\ref{defn_mathcalo}). Then $\forall \ x \in I_d,$ we have that 
\begin{equation} \label{derivtiveboundeq}
\mathcal{T}^{\prime}(x) \sim (d-c^{-1/4})^{-1}.
\end{equation}
\end{lem}

The following perturbation identity plays the key role in our proof, as it naturally provides us a way to incorporate the Green functions using a deterministic equation. Its proof can be found in \cite[Lemma 6.1]{KY2}.
\begin{lem}\label{lem_perb}  Recall (\ref{linearconstruction1}), assume $\mu \in \mathbb{R} / \bm{\sigma}(H)$ and $\det \mathbf{D} \neq 0$, then $\mu \in \bm{\sigma}(\tilde{H})$ if and only if
\begin{equation} \label{lem_perbubationequation}
\det(\mathbf{U}^*G(\mu)\mathbf{U}+\mathbf{D}^{-1})=0.
\end{equation}
\end{lem}
The following lemma establishes the connection between the Green functions of $H$ and $\tilde{H}$ defined in (\ref{linearconstruction1}), which is proved in the supplementary material \cite{DBS}. 
\begin{lem}\label{lem_gtitle} For $z \in \mathbb{C}^+,$ we have
\begin{equation}\label{defn_greenrep1}
\tilde{G}(z)=G(z)-G(z)\mathbf{U}(\mathbf{D}^{-1}+\mathbf{U}^*G(z)\mathbf{U})^{-1} \mathbf{U}^* G(z),
\end{equation}
and 
\begin{equation} \label{defn_greenrep2}
\mathbf{U}^* \tilde{G}(z) \mathbf{U}=\mathbf{D}^{-1}-\mathbf{D}^{-1}(\mathbf{D}^{-1}+\mathbf{U}^*G(z)\mathbf{U})^{-1}\mathbf{D}^{-1}.
\end{equation}
\end{lem}

One of the key ingredients of our computation are the local laws. We firstly introduce the anisotropic local law, which can be found in \cite[Theorem 3.6]{KY1}. Denote 
\begin{equation} \label{defn_psi}
\Psi(z):=\sqrt{\frac{\operatorname{Im}m_{2c}(z)}{N \eta}}+\frac{1}{N \eta},  \ \  \underline{\Sigma}:=\left( {\begin{array}{*{20}c}
   {z^{-1/2}} & 0  \\
   0 & {I}  \\
   \end{array}} \right),
\end{equation}
and $m(z) \equiv m_{N}(z)$ as the unique solution of the equation 
\begin{equation*}
f(m(z))=z, \ \operatorname{Im} m(z) \geq 0, \ f(x)=-\frac{1}{x}+\frac{1}{c_N} \frac{1}{x+1}.
\end{equation*}
Recall (\ref{green2}), the following lemma shows that $G(z)$ converges to a deterministic matrix $\Pi(z)$ with high probability. 
\begin{lem}\label{lem_anisotropic}  Fix $\tau > \epsilon_1$, then  for all $z \in \mathbf{D}(\tau)$, with $1-N^{-D_1}$ probability,  for any unit deterministic vectors $\mathbf{u}, \mathbf{v} \in \mathbb{R}^{M+N},$  we have 
\begin{equation}\label{thm_anisotropic_eq11}
|<\mathbf{u}, \underline{\Sigma}^{-1}(G(z)-\Pi(z)) \underline{\Sigma}^{-1}\mathbf{v}>|\leq N^{\epsilon_1} \Psi(z), \ |m_2(z)-m(z)|\leq \frac{N^{\epsilon_1}}{N\eta},
\end{equation}
where $\Pi(z)$ is defined as 
\begin{equation} \label{convergentlimit}
\Pi(z):=\left( {\begin{array}{*{20}c}
   { -z^{-1}(1+m(z))^{-1}} & 0  \\
   0 & {m(z)}  \\
   \end{array}} \right).
\end{equation}
\end{lem}
%{\color{red} say something here. In the original paper, the element in $\Pi$ is not exactly $m_{1c}(z), \ m_{2c}(z)$ but something depending on $N.$ {\bf The main issue here is that we have to discuss how close }}
It is notable that in general, $m(z)$ depends  on $N$ and Lemma \ref{lem_boundsm1m2} also holds for $m(z).$ However, in our computation, we can replace $m(z)$ with $m_{2c}(z)$ due to the following local MP law, which is proved in \cite[Theorem 3.1]{PY}. 
\begin{lem}\label{lem_locallaw}Fix $\tau > \epsilon_1$, then  for all $z \in \mathbf{D}(\tau)$, with $1-N^{-D_1}$ probability, we have
\begin{equation*}
|m_2(z)-m_{2c}(z)| \leq N^{\epsilon_1} \Psi(z).
\end{equation*}
\end{lem}

Beyond the support of the limiting spectrum of the MP law, we have stronger results all the way down to the real axis. More precisely, define the region
\begin{equation}\label{defn_tilded}
\tilde{\mathbf{D}}(\tau,\epsilon_1):=\{z \in \mathbb{C}^{+}:\lambda_+ +N^{-2/3+\epsilon_1} \leq E \leq \tau^{-1}, 0<\eta \leq \tau^{-1} \},
\end{equation}
then we have the following stronger control on $\tilde{\mathbf{D}}(\tau, \epsilon_1).$ The proof can be found in \cite[Theorem 3.12]{BEKYY} and  \cite[Theorem 3.7]{KY1}. 
\begin{lem} \label{lem_isotropic_outspec} For $z \in \tilde{\mathbf{D}}(\tau, \epsilon_1), $ with $1-N^{-D_1}$ probability, we have
\begin{equation*}
|<u,\mathcal{G}_2(z)v>-m_{2c}(z)<u,v>| \leq N^{-1/2+\epsilon_1}(\kappa+\eta)^{-1/4},
\end{equation*} 
for all unit vectors $u, v \in \mathbb{R}^N.$ Similar result holds for $\mathcal{G}_1(z), m_{1c}(z).$  Furthermore,  for any deterministic vectors $\mathbf{u}, \mathbf{v} \in \mathbb{R}^{M+N}$, we have 
\begin{equation} \label{thm_anisotropicoutspec}
|<\mathbf{u}, \underline{\Sigma}^{-1}(G(z)-\Pi(z))\underline{\Sigma}^{-1}\mathbf{v}>| \leq N^{-1/2+\epsilon_1}(\kappa+\eta)^{-1/4}.
\end{equation}
\end{lem}
Denote  the non-trivial classical eigenvalue locations $\gamma_1 \geq \gamma_2 \geq \cdots \geq \gamma_{K}$ of $XX^*$ as $\int_{\gamma_i}^{\infty} d \rho_{1c} =\frac{i}{N}$,  where $\rho_{1c}$ is defined in (\ref{mp_density}).  The consequent result of Lemma \ref{lem_anisotropic} is the rigidity of eigenvalues,  which can be found in \cite[Theorem 3.5]{BKYY}. 
\begin{lem}\label{lem_rigidity}
Fix any small $\tau \in (0,1),$ for $1 \leq i \leq (1-\tau)K,$  with $1-N^{-D_1}$ probability, we have
\begin{equation*}
|\lambda_{i}-\gamma_i| \leq N^{-2/3+\epsilon_1}(i\wedge (K+1-i))^{-1/3}.
\end{equation*}
Furthermore, if $c \neq 1,$ the above estimate holds for  all $i=1,2,\cdots, K.$
\end{lem}
Using Lemma \ref{lem_rigidity}, we find that $\kappa_j^d$ defined in (\ref{thm_non_left}) is a deterministic version of $\kappa_{\mu_j}=|\mu_j-\lambda_+|.$
%Denote $\xi_i \in \mathbb{R}^M, \zeta_i \in \mathbb{R}^N$ as the singular vectors of $X,$ the following isotropic delocalization bounds  implies that the entries $\xi_i(k), \zeta_i(k)$
%of the singular vectors are strongly oscillating in the sense that $| \sum \xi_i(k) | \prec 1$ but $\sum |\xi_i(k)| \succ N^{1/2},$ which implies the completely delocalization of the singular vectors.   Due to the purpose of our application, we only consider the singular vectors near the right edge.
%\begin{lem}[Isotropic delocalization]\label{lem_delocal}
%For $i \leq C,$ where $C>0$ is a large constant, for any normalized vector $m \in \mathbb{R}^M, n \in \mathbb{R}^N,$ with $1-N^{-D_1}$ probability, we have
%\begin{equation*}
%\max_{i}\{|<\xi_i,m>|^2+|<\zeta_i,n>|^2 \} \leq N^{-1+\epsilon_1}.
%\end{equation*}
%\end{lem}
\section{Proofs of Theorem \ref{thm_location} and \ref{thm_sigularvector} }\label{section_ev}
\subsection{Singular values}
In this subection, we focus on the singular values of $\tilde{S}$ and prove Theorem \ref{thm_location}. We will follow the basic idea of \cite{KY2} and slightly modify the proof.  A key deviation from their proof is that our matrix $\mathbf{D}$ defined in (\ref{defn_mathbfuv}) is not diagonal, it appears that in order to analyze (\ref{lem_perbubationequation}), they only need to deal with the diagonal elements but we need to control the whole matrix. We will make use of the following interlacing theorem  for rectangular matrices, the proof can be found in  \cite[Exercise 1.3.22]{Tao}.
\begin{lem}\label{lem_interlacing} For any $M \times N$ matrices $A, B$, denote $\sigma_i(A)$ as the $i$-th largest singular value of $A$, then we have
\begin{equation*}
\sigma_{i+j-1}(A+B) \leq \sigma_i(A)+\sigma_j(B), \ 1 \leq i,j, i+j-1 \leq K.
\end{equation*}
\end{lem}

 The proof  relies on two main steps: (i) fix a configuration independent of $N$, establish two permissible regions, $\Gamma(\mathbf{d})$ of $k^+$ components and $I_0,$ where the outliers of $\tilde{S} \tilde{S}^*$ are allowed to lie in $\Gamma(\mathbf{d})$ and each component contains precisely one eigenvalue and the $r-k^+$ non-outliers lie in $I_0$; (ii) a continuity argument where the result of (i) can be extended to arbitrary $N-$dependent $\mathbf{D}$.  

 The following $2r \times 2r$ matrix plays the key role in our analysis
 \begin{equation}\label{def_mrz}
M^r(x):=\mathbf{U}^*G(x)\mathbf{U}+\mathbf{D}^{-1}.
\end{equation}
By Lemma \ref{lem_perb}, $x \in \bm{\sigma}(\tilde{S}\tilde{S}^*)$ if and only if $\det M^r(z)=0.$ Using Lemma \ref{lem_locallaw} and \ref{lem_isotropic_outspec},
we find that $x^{-r} T^r(x) \approx \det M^r(x)$, where $T^r(x)$ is defined in (\ref{defn_ts}). As $T^r(x)$ behaves differently in $\Gamma(\mathbf{d})$ and $I_0,$ we will use different strategies to prove (\ref{thm_location_outlier}) and (\ref{thm_location_bulk}). 

We remark that, our discussion is slightly easier than \cite[Section 6]{KY2}, in particular the counting argument of the non-outliers. The reason is, for the application purpose, we only need the result of (\ref{thm_location_bulk}) to locate the eigenvalues around $\lambda_+.$ However, in \cite{KY2}, they have stronger results to stick the eigenvalues of $\tilde{S}\tilde{S}^*$ around those of  $XX^*.$ We will not pursue this generalization in this paper.

\begin{proof}[Proof of Theorem \ref{thm_location}] Denote $k^0:=r-k^+$ and write 
\begin{equation*}
\mathbf{d}=(d_1, \cdots, d_r)=( \mathbf{d}^0, \mathbf{d}^+), \ \mathbf{d}^{\sigma}=(d_1^{\sigma},\cdots, d_{k^{\sigma}}^{\sigma}), \ \sigma=0,+,
\end{equation*}
where we adapt the convention 
\begin{equation*}
d_{k^0}^0 \leq \cdots \leq d^0_{1} \leq c^{1/4}<d_{k^+}^+ \leq \cdots \leq d_{1}^+, \  k^0+k^+=r.
\end{equation*}
%Recall that $p(d)$ attains its minimum at $d=c^{-1/4},$ we denote that for any $d^0_i, d^+_j,$ we call them a {\bf couple} if 
%\begin{equation*}
%d^0_i d^+_j=\frac{1}{\sqrt{c}}, \ 1 \leq i \leq k^0, \ 1 \leq j \leq k^+.
%\end{equation*}
%By elementary calculation, we find that for each couple,
%\begin{equation*}
%p(d^0_i)=p(d^+_j).
%\end{equation*}
%We also denote 
%\begin{equation*}
%L_i^+:=\{p(d_j^0), j=1,2,\cdots, n_i\}, \ \text{where} \ d_j^0, j=1, \cdots, n_i  \ \text{are in a couple with} \ d_i^+.
%\end{equation*}
\noindent Next we define the sets 
\begin{equation} \label{defn_D+}
\mathcal{D}^+(\epsilon_0):=\{\mathbf{d}^+: c^{-1/4}+N^{-1/3+\epsilon_0} \leq d^+_i \leq \tau^{-1}, \ i=1, \cdots, k^+ \}, 
\end{equation}
\begin{equation}\label{defn_D0}
\mathcal{D}^0(\epsilon_0):=\{\mathbf{d}^0: 0<d_i^0 < c^{-1/4}+N^{-1/3+\epsilon_0},\ i=1, \cdots, k_0\},
\end{equation}
and the sets of allowed $\mathbf{d}'s,$ which is $
\mathcal{D}(\epsilon_0):= \{(\mathbf{d}^0, \mathbf{d}^+): \ \mathbf{d}^{\sigma} \in \mathcal{D}^{\sigma}(\epsilon_0), \ \sigma=+,0 \}.$ Denote the following sequence of intervals 
\begin{equation}\label{defn_i+d}
I_{i}^+(\mathbf{d}):=[p(d_i^+)-N^{-1/2+\epsilon_3}(d_i^+-c^{-1/4})^{1/2},  \ p(d_i^+)+N^{-1/2+\epsilon_3}(d_i^+-c^{-1/4})^{1/2}],
\end{equation}
where $\epsilon_3$ satisfies the following condition
\begin{equation}\label{conditionepsilon3}
C \epsilon_1 <\epsilon_3 <\frac{1}{4}\epsilon_0, \ C>2 \ \text{is some large constant.}
\end{equation} 
%Furthermore, we will say $d^+_i, d^0_j$ are in a {\bf class} if $p(d^0_j) \in m^+_i(\mathbf{d}).$ Denote
%\begin{equation*}
%C_i^+:=\{p(d_j^0), j=1,2,\cdots, n_i\}, \ \text{where} \ d_j^0, j=1, \cdots, n_i  \ \text{are in a class with} \ d_i^+.
%\end{equation*}
%To remove the influence of the class,  we now define the sets
%\begin{equation}\label{defn_i+d}
%I_{i}^+(\mathbf{d}):=m_i^+(\mathbf{d}) \cap (C_i^+)^c \cup (L_i^+).
%\end{equation}
For $\mathbf{d} \in \mathcal{D}(\epsilon_0),$ we denote $\Gamma(\mathbf{d}):= \cup_{i=1}^{k^+} I_i^+(\mathbf{d}) $ and $I^0:=[\lambda_+-N^{-2/3+C^{\prime}\epsilon_0}, \ \lambda_++N^{-2/3+C^{\prime}\epsilon_0}]$, where $C^{\prime}$ satisfies $2<C^{\prime}<4$.

For a first step, we show that $\Gamma(\mathbf{d})$ is our permissible region which keeps track of the outlier eigenvalues of  $\tilde{S} \tilde{S}^*.$  And the rest of the eigenvalues  corresponding to $\mathcal{D}^0(\epsilon_0)$ will lie in $I^0.$ We fix a configuration $\mathbf{d}(0) \equiv \mathbf{d}$ that is independent of $N$ in this step.
\begin{lem}\label{lem_permissble} For any $\mathbf{d} \in \mathcal{D}(\epsilon_0)$, with $1-N^{-D_1}$ probability, we have
\begin{equation}\label{lem_permissble_subset}
\bm{\sigma^+}(\tilde{S}\tilde{S}^*)  \subset \Gamma(\mathbf{d}),
\end{equation}
where $\bm{\sigma^+}(\tilde{S}\tilde{S}^*)$ is the set of the outlier eigenvalues of $\tilde{S}\tilde{S}^*$ associated with $\mathcal{D}^{+}(\epsilon_0).$ Moreover, each interval $I^+_i(\mathbf{d})$ contains precisely one eigenvalue of $\tilde{S} \tilde{S}^*, \ i=1,2, \cdots, k^+$. Furthermore, we have
\begin{equation}\label{lem_permissble_bulk}
\bm{\sigma^o}(\tilde{S}\tilde{S}^*) \subset I^0,
\end{equation}
where $\bm{\sigma^o}(\tilde{S}\tilde{S}^*)$ is the set of the non-outlier eigenvalues corresponding to $\mathcal{D}^0(\epsilon_0).$
\end{lem}
\begin{proof}
First of all, it is easy to check that $\Gamma(\mathbf{d}) \cap I^0=\emptyset$ using (\ref{defn_pprime}) and the fact $C^{\prime}>2$. Denote $S_b:=p(d_{k^+}^+)-N^{-1/2+\epsilon_3}(d_{k^+}^+-c^{-1/4})^{1/2}.$
In order to prove (\ref{lem_permissble_subset}), we first consider the case when $x > S_b.$ It is notable that $x \notin \bm{\sigma}(XX^*)$ by Lemma \ref{lem_rigidity}, (\ref{defn_pprime}) and (\ref{conditionepsilon3}). Recall (\ref{convergentlimit}) and (\ref{def_mrz}), using the fact $r$ is bounded and Lemma \ref{lem_isotropic_outspec}, with $1-N^{-D_1}$ probability, we have
\begin{equation}\label{matrix_estimation}
M^r(x)=\mathbf{U}^*\Pi(x)\mathbf{U}+\mathbf{D}^{-1}+O(N^{-1/2+\epsilon_1}\kappa^{-1/4}).
\end{equation} 
 It is well-known that if $\lambda \in \bm{\sigma}(A+B)$ then $\text{dist}(\lambda, \bm{\sigma}(A)) \leq ||B||; $ therefore, we have that $\mu_i(\tilde{S} \tilde{S}^*) \leq \tau^{-1}, i=1,\cdots, K $ for $\tau>0$  defined in (\ref{DOMAIN1}). Recall (\ref{defn_ts}), by (\ref{defn_pprime}),  (\ref{derivtiveboundeq})  and (\ref{conditionepsilon3}),  with $1-N^{-D_1}$ probability,  we have
\begin{equation} \label{outside_bound}
| T^r(x) | \geq  N^{-1/2+(C-1)\epsilon_1} \kappa^{-1/4},   \   \text{if} \ x  \in [S_b, \tau^{-1}] /\Gamma(\mathbf{d}).
\end{equation}
Using the formula
\begin{equation*}
\det
\begin{bmatrix}
 x I_r  & \text{diag}(\alpha_1, \cdots, \alpha_r) \\ 
 \text{diag}(\alpha_1, \cdots, \alpha_r) & y I_r 
\end{bmatrix}
=\prod_{i=1}^r (xy-\alpha_i^2),
\end{equation*}
Lemma \ref{lem_locallaw}, (\ref{relationm1m2}) and (\ref{matrix_estimation}), we conclude that
\begin{equation} \label{determinateboud}
\det (\mathbf{D}^{-1}+\mathbf{U}^*\Pi(x)\mathbf{U})=x^{-r} T^r(x)+O(N^{-1/2+\epsilon_1}\kappa^{-1/4}).
\end{equation}
By  (\ref{outside_bound}) and (\ref{determinateboud}), we conclude that $M^r(x)$ is non-singular when $x  \in [S_b, \tau^{-1}] / \Gamma(\mathbf{d}).$  

Next we will use Roch{\' e}'s theorem to show that inside the permissible region, each interval $I^+_i(\mathbf{d})$ contains precisely one eigenvalue of $\tilde{S} \tilde{S}^*$. Let $i \in \{1, \cdots, k^{+} \}$ and pick a small $N$-independent counterclockwise (positive-oriented) contour $\mathcal{C} \subset \mathbb{C} /[(1-c^{-1/2})^2, (1+c^{-1/2})^2] $ that encloses $p(d_i^+)$ but no other $p(d_j^+), \ j \neq i$. For large enough $N,$ define
$f(z):=\det(M^r(z)) , \ g(z):= \det(T^r(z)).$
By the definition of determinant, the functions $g, f$ are holomorphic on and inside $\mathcal{C}.$ And $g(z)$ has precisely one zero  $z=p(d_i^+)$ inside $\mathcal{C}.$ On $\mathcal{C}, $ it is easy to check that 
\begin{equation*}
\min_{z \in \mathcal{C}} |g(z)|\geq c>0,  \ |g(z)-f(z)|\leq N^{-1/2+\epsilon_1} \kappa^{-1/4},
\end{equation*}
where we use (\ref{matrix_estimation}) and Lemma \ref{lem_locallaw}. Hence, $f(z)$ has only one zero in $I_i^+(\mathbf{d})$ according to Rouch{\' e}'s theorem.  This concludes the proof of (\ref{lem_permissble_subset}) using Lemma \ref{lem_perb}. In order to prove (\ref{lem_permissble_bulk}), using the following fact: for any two $M \times N$ rectangular matrices $A, B$, we have $\sigma_i(A+B) \geq \sigma_i(A)+\sigma_K(B), \ i=1,\cdots, K,$
and Lemma \ref{lem_rigidity}, we find that
\begin{equation} \label{lowerboundforeiev}
\mu_{i} \geq \lambda_+-N^{-2/3+C^{\prime}\epsilon_0}, \ i= k^++1, \cdots, r. 
\end{equation}
For the non-outliers, we assume that $S_b>\lambda_++N^{-2/3+C^{\prime}\epsilon_0},$ otherwise the proof is already done. Now we assume $x \notin I_0,$ by (\ref{lem_permissble_subset}) and (\ref{lowerboundforeiev}), we only need to discuss the case when $x \in (\lambda_++N^{-2/3+C^{\prime}\epsilon_0}, \ S_b).$   In this case,  we will prove that $M^r(x)$ is non-singular by comparing with $M^r(z),$ where $z=x+iN^{-2/3-\epsilon_4}$ and $\epsilon_4<\epsilon_1$ is some small positive constant. Denote the spectral decomposition of $G(z)$ as 
\begin{equation*}
G(z) = \sum_{k} \frac{1}{\lambda_k-z} \mathbf{g}_{\alpha} \mathbf{g}_{\alpha}^*, \ \mathbf{g}_{\alpha} \in \mathbb{R}^{M+N}. 
\end{equation*}
Denote $\mathbf{u}_i, i=1,\cdots,2r $ as the $i$-th column in $\mathbf{U}$ defined in (\ref{defn_mathbfuv}) and abbreviate $\mathbf{u}_i^* G(z) \mathbf{u}_j$ as $G_{\mathbf{u}_i  \mathbf{u}_j}(z),$ and $\eta:=N^{-2/3-\epsilon_4},$  using spectral decomposition and the fact $x>\lambda_++N^{-2/3+C^{\prime}\epsilon_0},$ we have 
\begin{align*}
|G_{\mathbf{u}_i  \mathbf{u}_j}(x)-G_{\mathbf{u}_i  \mathbf{u}_j}(x+i\eta)| \leq \operatorname{Im}G_{\mathbf{u}_i  \mathbf{u}_i}(x+i\eta) +\operatorname{Im}G_{\mathbf{u}_j \mathbf{u}_j}(x+i\eta).
\end{align*}
Therefore, by Lemma \ref{lem_locallaw} and \ref{lem_isotropic_outspec}, with $1-N^{-D_1}$ probability, we have
\begin{equation*}
M^r(x)=M^r(z)+O( N^{\epsilon_1}\left(\operatorname{Im}m_{2c}(z)+\sqrt{\frac{\operatorname{Im}m_{2c}(z)}{N\eta}}\right)).
\end{equation*}
Using Lemma \ref{lem_boundsm1m2} and a similar discussion to (\ref{outside_bound}), we have
\begin{equation*}
M^r(x)=T^r(z)+O(N^{-1/3}(N^{-C^{\prime}\epsilon_0/4}+N^{\epsilon_1-C^{\prime}\epsilon_0/4})).
\end{equation*}
By Lemma \ref{lem_boundsm1m2} and \ref{lem_locallaw}, we find that 
$|T^r(z)| \geq N^{-1/3+\frac{C^{\prime}\epsilon_0}{2}},$
where we use the assumption that $x>\lambda_++N^{-2/3+C^{\prime}\epsilon_0}.$ Therefore, $M^r(x)$ is non-singular as we have assumed $2<C^{\prime}<4$. This concludes the proof of (\ref{lem_permissble_bulk}).
\end{proof}

In the second step, we will extend the proof to any configuration $\mathbf{d}(1)$ depending on $N$ using the continuity argument. This is done by a bootstrap argument by choosing a continuous path connecting $\mathbf{d}(0)$ and $\mathbf{d}(1)$. It is recorded as the following
lemma and its proof will be provided in the supplementary material \cite{DBS}.
\begin{lem}\label{lem_step2} For any $N$-dependent configuration $\mathbf{d}(1) \in \mathcal{D}(\epsilon_0),$ (\ref{thm_location_outlier}) and (\ref{thm_location_bulk}) hold true. 
\end{lem}

\end{proof}
\subsection{Singular vectors}\label{section_eve}
In this section, we focus on the local behavior of singular vectors. We will follow the discussion of \cite[Section 5 and 6]{BKYY}.  We first deal with the outlier singular vectors and then the non-outlier ones.  Due to similarity, we only prove (\ref{thm_right}) and (\ref{thm_non_right}),  (\ref{thm_left}) and (\ref{thm_non_left}) can be handled similarly.
\begin{proof}[Proof of (\ref{thm_right})] It is notable that,  by Lemma \ref{lem_isotropic_outspec} and  Theorem \ref{thm_location}, for $i \in \mathcal{O},$  there exists a constant $C>0,$ for $N$ large enough, with $1-N^{-D_1}$ probability , we can choose an event $\Xi$ such that for all $z \in \tilde{\mathbf{D}}(\tau, \epsilon_1)$ defined in (\ref{defn_tilded})
\begin{equation} \label{nonoverlappingcondition1}
\mathbf{1}(\Xi)|(V^{*}\mathcal{G}_2(z)V)_{ij}-m_{2c}(z)\delta_{ij}| \leq (\kappa+\eta)^{-1/4}N^{-1/2+C\epsilon_1}.
\end{equation}
Next we will restrict our discussion on the event $\Xi.$ Recall (\ref{defn_nuai}) and for $A \subset \mathcal{O}$, we define for each $i \in A$ the radius
\begin{equation}\label{defn_rho}
\rho_i:=\frac{\nu_i \wedge (d_i-c^{-1/4})}{2}.
\end{equation}
Under the assumption of  (\ref{defn1_nonoverlapping}), we have (see the equation (5.10) of \cite{BKYY})
\begin{equation}\label{boundrho}
\rho_i \geq \frac{1}{2} (d_i-c^{-1/4})^{-1/2}N^{-1/2+\epsilon_0}.
\end{equation} 
We define the contour $\Gamma:=\partial \Upsilon$  as the boundary of the union of discs $\Upsilon:=\cup_{i \in A} B_{\rho_i}(d_i)$, where $B_{\rho}(d)$ is the open disc of radius $\rho$ around $d.$ We summarize the basic properties of $\Upsilon$ as the following lemma, its proof can be found in  \cite[Lemma 5.4 and 5.5]{BKYY}.  
\begin{lem} \label{lem_contour} Recall (\ref{defn_pd}) and (\ref{defn_tilded}),  we have
$\overline{p(\Upsilon)} \subset \tilde{\mathbf{D}}(\tau, \epsilon_1).$ Moreover, each outlier $\{\mu_i \}_{i \in A}$ lies in $p(\Upsilon),$ and all the other eigenvalues of $\tilde{S}\tilde{S}^*$ lie in the complement of $\overline{p(\Upsilon)}.$
\end{lem}

Armed with the above results, we now start the proof of the outlier singular vectors. Our starting point is an integral representation of the singular vectors.  By (\ref{green2}), we have
\begin{equation}\label{rightsingularvector}
v_i^{*}\tilde{\mathcal{G}}_2v_j=\mathbf{v}_{i}^{*}\tilde{G} \mathbf{v}_{j}, 
\end{equation}
where $\mathbf{v}_i \in \mathbb{R}^{M+N}$ is the natural embedding of $v_i$ with $\mathbf{v}_i=(0, v_i)^*$. Recall (\ref{defn_projection}), using the spectral decomposition of $\tilde{\mathcal{G}}_2(z),$ Lemma \ref{lem_contour} and Cauchy's integral formula, we have
\begin{equation} \label{eigenvectorrepresentation}
\mathbf{P}_{r} =-\frac{1}{2 \pi i}\int_{p(\Gamma)} \tilde{\mathcal{G}}_2(z)dz=-\frac{1}{2 \pi i} \int_{\Gamma} \mathcal{\tilde{G}}_2(p(\zeta))p^{\prime}(\zeta)d\zeta.
\end{equation}
By Lemma   \ref{lem_gtitle},  Cauchy's integral formula, (\ref{rightsingularvector}) and (\ref{eigenvectorrepresentation}), we have
\begin{align}\label{intergralrepresentationv}
<v_i, \mathbf{P}_rv_j> = \frac{1}{2 d_i d_{j} \pi i} \int_{p(\Gamma)} (\mathbf{D}^{-1}+\mathbf{U}^*G(z)\mathbf{U})^{-1}_{i j} \frac{dz}{z},
\end{align}
where $\bar{i}, \ \bar{j}$ are defined as $\bar{i}:=r+i, \ \bar{j}:=r+j.$ Recall (\ref{convergentlimit}), as $\mathbf{D}^{-1}+\mathbf{U}^*\Pi(z)\mathbf{U}$ is of finite dimension, by Lemma \ref{lem_locallaw}, \ref{lem_isotropic_outspec},  (\ref{relationm1m2}) and (\ref{nonoverlappingcondition1}), we can now use $\Pi(z)$ as
\begin{equation*}
\Pi(z):=\left( {\begin{array}{*{20}c}
   m_{1c}(z)& 0  \\
   0 & {m_{2c}(z)}  \\
   \end{array}} \right).
\end{equation*}
Next we decompose $\mathbf{D}^{-1}+\mathbf{U}^*G(z)\mathbf{U}$ by
\begin{equation} \label{denumeratordecomposition}
\mathbf{D}^{-1}+\mathbf{U}^*G(z)\mathbf{U}=\mathbf{D}^{-1}+\mathbf{U}^*\Pi(z)\mathbf{U}-\Delta(z), \ \Delta(z)=\mathbf{U}^*\Pi(z)\mathbf{U}-\mathbf{U}^*G(z)\mathbf{U}.
\end{equation}
It is notable that $\Delta(z)$ can be controlled by  Lemma \ref{lem_locallaw} and \ref{lem_isotropic_outspec}. Using the resolvent expansion to the order of one on (\ref{denumeratordecomposition}),  we have 
\begin{equation}\label{sij_decomposition}
<v_i, \mathbf{P}_r v_j>=\frac{1}{d_i d_j}(S^{(0)}+S^{(1)}+S^{(2)}),
\end{equation}
where 
\begin{equation*}
S^{(0)}:= \frac{1}{2 \pi i} \int_{p(\Gamma)}(\frac{1}{\mathbf{D}^{-1}+\mathbf{U}^*\Pi(z)\mathbf{U}})_{i j}\frac{dz}{z},
\end{equation*} 
\begin{equation*}
S^{(1)}=\frac{1}{2 \pi i} \int_{p(\Gamma)}[\frac{1}{\mathbf{D}^{-1}+\mathbf{U}^*\Pi(z)\mathbf{U}}\Delta(z)\frac{1}{\mathbf{D}^{-1}+\mathbf{U}^*\Pi(z)\mathbf{U}}]_{i j}\frac{dz}{z},
\end{equation*}
\begin{equation*}\label{defn_sij2}
S^{(2)}=\frac{1}{2 \pi i} \int_{p(\Gamma)}[\frac{1}{\mathbf{D}^{-1}+\mathbf{U}^*\Pi(z)\mathbf{U}}\Delta(z)\frac{1}{\mathbf{D}^{-1}+\mathbf{U}^*\Pi(z)\mathbf{U}}\Delta(z)\frac{1}{\mathbf{D}^{-1}+\mathbf{U}^*G(z)\mathbf{U}}]_{i j}\frac{dz}{z}.
\end{equation*}
By an elementary computation, we have
\begin{equation} \label{biginverse}
(\mathbf{D}^{-1}+\mathbf{U}^*\Pi(z)\mathbf{U})^{-1}_{i j}=
\begin{cases}
\delta_{ij}\frac{zm_{2c}(z)}{z m_{1c}(z)m_{2c}(z)-d_i^{-2}}, & 1 \leq i, j \leq r; \\
\delta_{ij}\frac{zm_{1c}(z)}{z m_{1c}(z)m_{2c}(z)-d_i^{-2}}, & r \leq i, j \leq 2r; \\
\delta_{\bar{i} j} (-1)^{i+j} \frac{z^{1/2}d_{i}^{-1}}{zm_{1c}(z)m_{2c}(z)-d_{i}^{-2}},& 1 \leq i \leq r, \ r \leq j \leq 2r ;\\
\delta_{i \bar{j}}(-1)^{i+j} \frac{z^{1/2}d_j^{-1}}{zm_{1c}(z)m_{2c}(z)-d_j^{-2}} , & r \leq i \leq 2r, \ 1 \leq j \leq r. 
\end{cases}
\end{equation}
Using the fact $ p_i m_{1c}(p_i)m_{2c}(p_i)=\frac{1}{d_i^2}$ and the residual theorem, we have
\begin{equation}\label{s0boundfinal}
S^{(0)}= \delta_{ij} \frac{m_{2c}(p_i)}{\mathcal{T}^{\prime}(p_i)}=\delta_{ij}\frac{d_i^4-c^{-1}}{d_i^2+1}.
\end{equation}
Next we control the term $S^{(1)}.$ Applying (\ref{biginverse}) on $S^{(1)}$, we have 
%\begin{align*}\label{inversecomputation2}
%[& \frac{1}{\mathbf{D}^{-1}+\mathbf{U}^*\Pi(z)\mathbf{U}}\Delta(z)\frac{1}{\mathbf{D}^{-1}+ \mathbf{U}^*\Pi(z)\mathbf{U}}]_{i j} \nonumber \\
%=&\frac{zm_{2c}(z)}{zm_{1c}(z)m_{2c}(z)-d_j^{-2}}\left[\frac{zm_{2c}(z)}{zm_{1c}(z)m_{2c}(z)-d_i^{-2}} \Delta(z)_{i j}+ \frac{(-1)^{i+\bar{i}}z^{1/2}d_i^{-1}}{zm_{1c}(z)m_{2c}(z)-d_i^{-2}}\Delta(z)_{\bar{i} j}\right] \nonumber \\
%+ & \frac{(-1)^{j+\bar{j}}z^{1/2}d_j^{-1}}{zm_{1c}(z)m_{2c}(z)-d_j^{-2}}\left[(\frac{zm_{2c}(z)}{zm_{1c}(z)m_{2c}(z)-d_i^{-2}} \Delta(z)_{i \bar{j}}+ \frac{(-1)^{i+\bar{i}}z^{1/2}d_i^{-1}}{zm_{1c}(z)m_{2c}(z)-d_i^{-2}}\Delta(z)_{\bar{i} \bar{j}}\right]. 
%\end{align*}
%Therefore, we can write 
\begin{equation} \label{defn_sij1}
S^{(1)}=\frac{1}{2 \pi i} \int_{p(\Gamma)} \frac{f(z)}{(zm_{1c}(z)m_{2c}(z)-d_i^{-2})(zm_{1c}(z)m_{2c}(z)-d_j^{-2})}dz, 
\end{equation}
where $f(z)=f_1(z)+f_2(z)$ and $f_{1,2}(z)$ are defined as
\begin{equation*}
f_1(z):=m_{2c}(z)[zm_{2c}(z)\Delta(z)_{ij}+(-1)^{i+\bar{i}}z^{1/2}d_i^{-1}\Delta(z)_{\bar{i} j}],
\end{equation*}
\begin{equation*}
f_2(z):=d_j^{-1}[(-1)^{j+\bar{j}}z^{1/2}m_{2c}(z)\Delta(z)_{i\bar{j}}+(-1)^{i+j+\bar{i}+\bar{j}}d_i^{-1}\Delta(z)_{\bar{i}\bar{j}}].
\end{equation*}
We now use the change of variable as in (\ref{eigenvectorrepresentation}) and rewrite $S^{(1)}$ as
\begin{equation*}
S^{(1)}=\frac{1}{2 \pi i} \int_{\Gamma} \frac{f(p(\zeta))}{(\zeta^{-2}-d_i^{-2})(\zeta^{-2}-d_j^{-2})}p^{\prime}(\zeta)d\zeta=d_i^2 d_j^2\frac{1}{2\pi i} \int_{\Gamma} \frac{f(p(\zeta))\zeta^4}{(d_i^2-\zeta^2)(d_j^2-\zeta^2)} p^{\prime}(\zeta)d\zeta,
\end{equation*}
where we use the fact $p(\zeta)m_{1c}(p(\zeta))m_{2c}(p(\zeta))=\zeta^{-2}.$ 
By (\ref{defn_pprime}),  Lemma \ref{lem_boundsm1m2} and \ref{lem_isotropic_outspec}, we conclude that
\begin{equation} \label{fderivativebound}
| f(p(\zeta))p^{\prime}(\zeta) \zeta^4| \leq (\zeta-c^{-1/4})^{1/2}N^{-1/2+\epsilon_1}.
\end{equation} 
Denote 
\begin{equation*}
f_{ij}(\zeta)=\frac{f(p(\zeta))p^{\prime}(\zeta) \zeta^4}{(d_i+\zeta)(d_j+\zeta)}.
\end{equation*}
As $f_{ij}$ is holomorphic inside the contour $\Gamma,$ by Cauchy's differentiation formula, we have
\begin{equation}\label{fcauchydiff}
f_{ij}^{\prime}(\zeta)=\frac{1}{2 \pi i} \int_{\mathcal{C}} \frac{f_{ij}(\xi)}{(\xi-\zeta)^2} d\xi, 
\end{equation}
where the contour $\mathcal{C}$ is the circle of radius $\frac{|\zeta-c^{-1/4}|}{2}$ centered at $\zeta.$ Hence, by (\ref{defn_pprime}), (\ref{fderivativebound}), (\ref{fcauchydiff}) and the residual theorem,  we have
\begin{equation}\label{fijderivative}
|f^{\prime}_{ij}(\zeta)| \leq (\zeta-c^{-1/4})^{-1/2} N^{-1/2+\epsilon_1}.
\end{equation}
In order to estimate $S^{(1)},$ we consider the following three cases (i) $i, j \in A, $ (ii) $i \in A, j \notin A,$ (or $i \notin A, \ j \in A$), (iii) $i, j \notin A $. By the residual theorem,  $S^{(1)}=0$ when case (iii) happens. Hence, we only need to consider the cases (i) and (ii).  For the case (i),  when $i \neq j$, by the residual theorem and (\ref{fijderivative}), we have
\begin{equation*}
|S^{(1)}|= d^2_i d^2_j  \left | \frac{f_{ij}(d_i)-f_{ij}(d_j)}{d_i-d_j} \right | \leq \frac{d^2_i d^2_j }{|d_i-d_j|} \left | \int_{d_i}^{d_j} |f^{\prime}_{ij}(t)|dt \right | \leq \frac{d^2_i d^2_jN^{-1/2+\epsilon_1} }{(d_i-c^{-1/4})^{1/2}+(d_j-c^{-1/4})^{1/2}} . 
\end{equation*}
When $i=j$, by the residual theorem, we have
$|S^{(1)}| \leq  d_i^4  (d_i-c^{-1/4})^{-1/2}N^{-1/2+\epsilon_1}. $ For the case (ii), when $i \in A,  j \notin A$, by the residual theorem and (\ref{nonoverlappingcondition1}), we have 
\begin{equation*}
|S^{(1)}|=|\frac{d^2_i d^2_j f_{ij}(d_i)}{d_i-d_j}| \leq  \frac{d^2_i d^2_j (d_i-c^{-1/4})^{1/2}}{|d_i-d_j|}N^{-1/2+\epsilon_1}.
\end{equation*}
We can get similar results when $i \notin A, j \in A.$ Putting all the cases together,  we find that
\begin{align} \label{sij1boundfinal}
|S^{(1)}|\leq N^{-1/2+\epsilon_1} \left[ \frac{\mathbf{1}(i \in A, j \in A) d^2_i d^2_j}{(d_i-c^{-1/4})^{1/2}+(d_j-c^{-1/4})^{1/2}} + \mathbf{1}(i \in A, j \notin A)\frac{d^2_i d^2_j (d_i-c^{-1/4})^{1/2}}{|d_i-d_j|} \nonumber \right. \\
 \left.  +  \mathbf{1}(i \notin A, j \in A)\frac{d^2_i d^2_j (d_j-c^{-1/4})^{1/2}}{|d_i-d_j|} \right].
\end{align}
Finally, we need to estimate $S^{(2)}.$ Here the residual calculations can not be applied directly as $\mathbf{U}^*G(z)\mathbf{U}$ is not necessary to be diagonal and  a relation comparable to $p(\zeta)m_{1c}(p(\zeta))m_{2c}(p(\zeta))=\zeta^{-2}$ does not exist.  Instead, we need to precisely choose  the contour $\Gamma.$ We record the result as the following lemma, whose proofs will be given in the supplementary material \cite{DBS}. 
\begin{lem}\label{lem_sij2bound} When $N$ is large enough, with $1-N^{-D_1}$ probability, for some constant $C>0,$ we have 
\begin{equation}\label{sij2boundfinal}
|S^{(2)}| \leq CN^{-1+2\epsilon_1}(\frac{1}{\nu_i}+\frac{\mathbf{1}(i \in A)}{|d_i-c^{-1/4}|})(\frac{1}{\nu_j}+\frac{\mathbf{1}(j \in A)}{|d_j-c^{-1/4}|}).
\end{equation} 
\end{lem}

Therefore, plugging (\ref{s0boundfinal}), (\ref{sij1boundfinal}) and (\ref{sij2boundfinal}) into (\ref{sij_decomposition}), we conclude the proof of (\ref{thm_right}). Before concluding this subsection, we briefly discuss the proof of (\ref{thm_left}).  By Lemma   \ref{lem_gtitle} and Cauchy's integral formula , we have
\begin{align*}
<u_i, \mathbf{P}_l  u_j> =\frac{1}{2 d_i d_{j} \pi i} \int_{p(\Gamma)} (D^{-1}+\mathbf{U}^*G(z)\mathbf{U})^{-1}_{\bar{i} \bar{j}} \frac{dz}{z}.
\end{align*}
Then we can use a similar discussion as (\ref{sij_decomposition}), computing the convergent limit from $S^{(0)}$ and controlling the bounds for $S^{(1)}$ and $S^{(2)}.$  We remark that the convergent limit is different because we  use  $(\mathbf{D}^{-1}+\mathbf{U}^* \Pi(z) \mathbf{U})_{ij}, r \leq i,j \leq 2r$ in  (\ref{biginverse}), which results in 
\begin{equation*}
S^{(0)}=\delta_{ij}\frac{m_{1c}(p_i)}{\mathcal{T}^{\prime}(p_i)}=\delta_{ij}\frac{d_i^4-c^{-1}}{d_i^2+c^{-1}}.
\end{equation*}
This concludes the proof of (\ref{thm_left}).
\end{proof}

For the non-outliers, the proof strategy for the outlier singular vectors will not work as we cannot use the residual theorem. We will use a spectral decomposition for our proof.  
\begin{proof}[Proof of (\ref{thm_non_right})]  Denote 
\begin{equation}\label{est_defn_eta}
z=\mu_j+i\eta,
\end{equation}
where $\eta$ is defined as the smallest solution of
\begin{equation}\label{defn_etamm}
\operatorname{Im} m_{2c}(z)=N^{-1+6\epsilon_1}\eta^{-1}.
\end{equation}
As we assume $j \leq (1-\tau)K$ or $c \neq 1,$ we conclude that $|z|$ has a constant lower bound. Therefore, by Lemma \ref{lem_anisotropic}, \ref{lem_locallaw}  and \ref{lem_isotropic_outspec}, with $1-N^{-D_1}$ probability,  we have
\begin{equation}\label{outlier_anisotropic}
|<\mathbf{u}, \underline{\Sigma}^{-1}(G(z)-\Pi(z))\underline{\Sigma}^{-1}\mathbf{v}>| \leq \frac{N^{4\epsilon_1}}{N \eta}.
\end{equation}
Recall (\ref{defn_kappa}), abbreviating $\kappa=|\mu_j-\lambda_+|, $ by Lemma \ref{lem_boundsm1m2} and (\ref{thm_location_bulk}), we find that (see \cite[(6.5) and (6.6)]{BKYY})
\begin{equation}\label{estimationofeta}
\eta \sim
\begin{cases}
\frac{N^{6\epsilon_1}}{N \sqrt{\kappa}+N^{2/3+2\epsilon_1}},& \text{if} \ \mu_j \leq \lambda_++N^{-2/3+4\epsilon_1}, \\
N^{-1/2+3\epsilon_1}\kappa^{1/4}, & \text{if} \ \mu_j\geq \lambda_++N^{-2/3+4\epsilon_1}.
\end{cases}.
\end{equation}
For $z$ defined in (\ref{est_defn_eta}),  by the spectral decomposition, we have 
\begin{equation}\label{spec_nonout}
<v_i, \tilde{v}_j>^2 \leq \eta <v_i, \operatorname{Im} \tilde{\mathcal{G}}_2(z) v_i>=\eta <\mathbf{v}_i, \operatorname{Im}\tilde{G}(z) \mathbf{v}_i>,
\end{equation}  
where $\mathbf{v}_i \in \mathbb{R}^{M+N}$ is the natural embedding of $v_i.$  By Lemma \ref{lem_gtitle}, we have 
\begin{equation*}
<\mathbf{v}_i, \tilde{G}(z)\mathbf{v}_i>=-\frac{1}{zd_i^2}(\mathbf{D}^{-1}+\mathbf{U}^* G(z)\mathbf{U})^{-1}_{ii}.
\end{equation*}
Similar to (\ref{sij_decomposition}), using a simple resolvent expansion and  (\ref{biginverse}) , we have
\begin{align} \label{bigexpansionnonoutlier}
<\mathbf{v}_i, & \tilde{G}(z)\mathbf{v}_i>  \nonumber  \\
= -\frac{1}{zd_i^2} & [ \frac{zm_{2c}(z)}{z m_{1c}(z)m_{2c}(z)-d^{-2}_{i}}+\frac{zf(z)}{(zm_{1c}(z)m_{2c}(z)-d_i^{-2})^2} \nonumber \\
& +  \left( [ (\mathbf{D}^{-1}+\mathbf{U}^*\Pi(z) \mathbf{U})^{-1} \Delta(z) ]^2 (\mathbf{D}^{-1}+ \mathbf{U}^*G(z)\mathbf{U})^{-1} \right)_{ii}] ,
\end{align}
where $f(z)$ is defined in (\ref{defn_sij1}). To estimate the right-hand side of (\ref{bigexpansionnonoutlier}), we  use the following error estimate
\begin{equation*}
\min_j |d_j^{-2}-\mathcal{T}(z)| \geq \operatorname{Im} \mathcal{T}(z) \sim \operatorname{Im} m_{2c}(z)=\frac{N^ {6\epsilon_1}}{N \eta} \gg \frac{N^{4\epsilon_1}}{N \eta} \geq |\Delta(z)|, 
\end{equation*}
where we use (\ref{outlier_anisotropic}) and Lemma \ref{lem_locallaw}. By a similar resolvent expansion, there exists some constant $C>0,$ such that
\begin{equation*}
\left | \left | \frac{1}{\mathbf{D}^{-1}+\mathbf{U}^*G(z)\mathbf{U}}  \right | \right | \leq \frac{C}{\operatorname{Im} m_{2c}(z)}=CN^{1-6\epsilon_1}\eta.
\end{equation*}
We therefore get from (\ref{bigexpansionnonoutlier}), the definition of $f$ and (\ref{outlier_anisotropic}) that
\begin{equation} \label{bound_gvv}
<\mathbf{v}_i, \tilde{G}(z)\mathbf{v}_i> =\frac{m_{2c}(z)}{1-d_i^2\mathcal{T}(z)}+O(\frac{d_i^2}{|1-d_i^2 \mathcal{T}(z)|^2} \frac{N^{4\epsilon_1}}{N \eta}).
\end{equation}
By (\ref{spec_nonout}), we have
\begin{align}\label{boundform}
<v_i, \tilde{v}_j>^2 \leq \frac{\eta}{|1-d^2_i\mathcal{T}(z)|^2} \left[ \operatorname{Im} m_{2c}(z) (1-d^2_ic^{1/2}+ \operatorname{Re}(d_i^2c^{1/2}-d_i^2 \mathcal{T}(z))) \right. \nonumber \\
\left. +d_i^2\operatorname{Re} m_{2c}(z) \operatorname{Im}\mathcal{T}(z)+\frac{Cd_i^2 N^{4\epsilon_1}}{N\eta} \right].
\end{align}
By (\ref{lem_realboundt}), (\ref{defn_etamm}) and (\ref{estimationofeta}), we have 
\begin{align*}
\operatorname{Im} m_{2c}(z)[(1-d_i^2c^{1/2})+& \operatorname{Re}(d_i^2c^{1/2}-d_i^2 \mathcal{T}(z))] \\
& \leq \frac{C N^{6 \epsilon_1}}{N \eta}\left(|d_i-c^{-1/4}|+\max\{\sqrt{\kappa+\eta}, \frac{\eta}{\sqrt{\kappa+\eta}}+\kappa\}\right).
\end{align*}
For the other item, by Lemma \ref{lem_boundsm1m2}, we have $|\operatorname{Re}m_{2c}(z) \operatorname{Im} \mathcal{T}(z)| \sim \operatorname{Im}m_{2c}(z).$ Putting all these estimates together, we have
\begin{equation*}
<v_i, \tilde{v}_j>^2 \leq \frac{CN^{6\epsilon_1}}{N |1-d_i^2\mathcal{T}(z)|^2}.
\end{equation*}
The rest of the proof leaves to give an estimate of  $1-d_i^2\mathcal{T}(z).$ We summarize it as the following lemma and put its proof in the supplementary material \cite{DBS}.
\begin{lem}\label{lem_bound1dit}
Recall (\ref{mp_density}), for all $\mu_j \in [\lambda_{-}, \lambda_{+}+N^{-2/3+C\epsilon_0}], $ there exists a constant $\delta>0,$ such that 
\begin{equation*}
|1-d_i^2\mathcal{T}(z)| \geq \delta d_i^2(|d_i^{-2}-c^{1/2}|+\operatorname{Im} \mathcal{T}(z)).
\end{equation*} 
\end{lem}
Therefore, we have 
\begin{equation*}
<v_i, \tilde{v}_j>^2 \leq \frac{N^{C\epsilon_0}}{N((d_i-c^{-1/4})^2+\kappa^d_j)}, \ \kappa_j^d:=N^{-2/3}(j \wedge (K+1-j))^{2/3},
\end{equation*}
where we use the fact that $\operatorname{Im} \mathcal{T}(z) \geq c\sqrt{\kappa_j^d}$ (see the equation (6.14) of  \cite{BKYY}).  This concludes the proof of (\ref{thm_non_right}). For the proof of  (\ref{thm_non_left}),  we will use the spectral decomposition
\begin{equation*}
<u_i, \tilde{u}_j>^2 \leq \eta <u_i, \operatorname{Im} \tilde{\mathcal{G}}_1(z) u_i>=\eta <\mathbf{u}_i, \operatorname{Im}\tilde{G}(z) \mathbf{u}_i>,
\end{equation*}  
and
\begin{equation*} 
<\mathbf{u}_i, \tilde{G}(z)\mathbf{u}_i>=-\frac{1}{zd_i^2}(\mathbf{D}^{-1}+\mathbf{U}^* G(z)\mathbf{U})^{-1}_{\bar{i} \bar{i}}.
 \end{equation*}
Then by the resolvent expansion similar to (\ref{bigexpansionnonoutlier}) and control the items using Lemma \ref{lem_boundsm1m2},  \ref{lem_anisotropic}, \ref{lem_locallaw} and \ref{lem_isotropic_outspec}, we can conclude the proof. 
\end{proof}

\paragraph{Acknowledgements.}
The author would like to thank  Zhigang Bao, Jeremy Quastel, B{\' a}lint Vir{\' a}g, Ke Wang and Zhou Zhou for fruitful discussions
and valuable suggestions, which have greatly improved the paper. The author  is also grateful to an anonymous referee, the associated editor and editor for providing detailed and constructive suggestions and comments, which have improved the paper significantly.

%\begin{supplement} 
%%\sname{Supplement A}
%\stitle{Supplement to "High dimensional deformed rectangular matrices with applications in matrix denoising"}
%\sdescription{This supplementary
%material contains auxiliary lemmas and proofs of  Proposition \ref{prop_choiseoutlier}, Theorem \ref{thm_sparse} and \ref{thm_rie}, Lemma \ref{lem_zm1zm2z}, \ref{lem_derivativebound}, \ref{lem_gtitle}, \ref{lem_step2}, \ref{lem_sij2bound} and \ref{lem_bound1dit}.}
%\end{supplement}

\paragraph{Supplementary material}\label{supplement}

{\bf{Supplement to "High dimensional deformed rectangular matrices with applications in matrix denoising"}}: This supplementary
material contains auxiliary lemmas and proofs of  Proposition \ref{prop_choiseoutlier}, Theorem \ref{thm_sparse} and \ref{thm_rie} , Lemma \ref{lem_zm1zm2z}, \ref{lem_derivativebound}, \ref{lem_gtitle}, \ref{lem_step2}, \ref{lem_sij2bound} and \ref{lem_bound1dit}.

\section*{Supplement to "High dimensional deformed rectangular matrices with applications in matrix denoising"}

\begin{proof}[Proof of Proposition 3.3] For $i \leq k^+,$ under the assumptions of Theorem 2.2 and  Assumption 3.2 of the paper, with $1-o(1)$ probability, we have that 
\begin{equation*}
\frac{\mu_i}{\mu_{i+1}}-1=\frac{\mu_i-\mu_{i+1}}{\mu_{i+1}}=O\Big(p(d_i)-p(d_{i+1})+N^{-1/2+C\epsilon_0} \Big).
\end{equation*}
Using Assumption 3.2 of the paper, with $1-o(1)$ probability,  we have that 
\begin{equation}\label{eq_r1}
\frac{\mu_i}{\mu_{i+1}}-1=O(1), \ i \leq k^+.
\end{equation}
And for $i=k^++1,$ with $1-o(1)$ probability, we have that
\begin{equation}\label{eq_r2}
\frac{\mu_i}{\mu_{i+1}}-1=O(N^{-2/3+C\epsilon_0}).
\end{equation} 
By definition, we have that
\begin{align}\label{eq_r3}
\mathbb{P}(q=k^+)&=\mathbb{P} \left( \bigcap_{1 \leq j \leq k^+} \{ \mathcal{R}_j>1+\tau \} \cap \{\mathcal{R}_{k^++1} \leq 1+\tau \}  \right) \nonumber \\
& \geq 1-\sum_{j=1}^{k+} \mathbb{P}(\mathcal{R}_j \leq 1+\tau)-\mathbb{P}(\mathcal{R}_{k^++1}>1+\tau).
\end{align}
Under Assumption 3.2 and the fact that $\epsilon_0$ is sufficiently small, for $\tau=O(N^{-2/3+(C+1)\epsilon_0}),$ we can conclude our proof using (\ref{eq_r1}), (\ref{eq_r2}) and (\ref{eq_r3}).
\end{proof}

\begin{proof}[Proof of Theorem 3.4]
Denote $S_1=\sum_{i=1}^{k^+}d_iu_iv^*_i, \ S_2=\sum_{i=k^++1}^{r}d_iu_iv^*_i,$  we have
\begin{equation*}
|| \hat{S}-S ||_F \leq ||\hat{S}-S_1 ||_F+\sqrt{\sum_{i=k^++1}^r d_i^2}.
\end{equation*}
It is easy to check that 
\begin{equation} \label{decomposeerror}
||\hat{S}-S_1||_F^2 \leq 2\sum_{i=1}^{k^+} (\hat{d}_i-d_i)^2+2 \operatorname{Tr} \left( RR^* \right),
\end{equation}
where $R$ is defined as  $R:=\sum_{i=1}^{k^+} \hat{d}_i \hat{u}_i \hat{v}_i^*-\sum_{i=1}^{k^+} \hat{d}_i u_i v_i^*.$ The first term on the right-hand side of (\ref{decomposeerror}) is bounded by $N^{-1+C\epsilon_0}$ using equation (2.8) of the paper.
%For the control of $R,$ we decompose it by
%\begin{equation*}
%R=\sum_{i=1}^{k^+}\hat{d}_i \hat{u}_i(\hat{v}_i^*-v_i^*)+\sum_{i=1}^{k^+} \hat{d}_i(\hat{u}_i-u_i)v_i^*.
%\end{equation*} 
For the second term, we only need to control $\operatorname{Tr}((\hat{v}_i-v_i)(\hat{v}_i-v_i)^*)$ and $\operatorname{Tr}((\hat{u}_i-u_i)(\hat{u}_i-u_i)^*)$   by Cauchy-Schwarz inequality. Due to similarity, we only prove for the right singular vectors.

Under the sparsity assumption, the non-zero entries of $S$ are confined on a block matrix $S_b$ of some fixed dimension $m \times n.$ Denote $\hat{S}_b:=S_b+X_b,$  if our algorithm can correctly choose the positions of the non-zero entries of $u_i, v_i$ (i.e. $\hat{S}_b$) with $1-o(1)$ probability, we can conclude our proof using the  fact (see \cite[Lemma 4.3]{SS})
\begin{equation*}
V_b=\hat{V}_b+O(||X_b^*X_b+S_b^*X_b+X_b^* S_b ||_F),
\end{equation*}
where $V_b, \hat{V}_b$ are the right singular vectors of $S_b, \hat{S}_b$ respectively. Therefore,
under the assumption that $x_{ij}$ is of variance $1/N, $ we have that with $1-o(1)$ probability, $V_b=\hat{V}_b+O(N^{-1/2+C\epsilon_0}).$ This concludes our proof. 

The rest of the proof leaves to show that equation (3.3) of the paper can correctly find the positions of the non-zero entries (i.e. $\hat{S}_b$) with $1-o(1)$ probability, which is summarized as the following lemma and we will put its proof in the supplementary material.
\begin{lem}\label{rem_kmeans} For $i=1,2,\cdots, k^+,$ denote $\mathcal{J}_i$ as the index set of the non-zero entries of $v_i,$ for some constant $C>0,$ there exists some $\delta \in (C\epsilon_0, \frac{1}{2}),$ with $1-o(1)$ probability, we have 
\begin{equation*}
|\tilde{v}_i(k)| \geq N^{-1/2+\delta}, \ k \in \mathcal{J}_i; \  |\tilde{v}_i(k)| \leq N^{-1/2+C\epsilon_0}, \ k \in \mathcal{J}_i^c \cap \{1,\cdots,N\}.
\end{equation*}
\end{lem}
\noindent By Lemma \ref{rem_kmeans}, we have that with $1-o(1)$ probability, $\max_{k_1 \notin \mathcal{J}_i} |\tilde{v}_i(k_1)| \ll \min_{k_2 \in \mathcal{J}_i}|\tilde{v}_i(k_2)|,$ which implies that Algorithm 1 can correctly recover the sparse structure of the singular vectors.
Finally, we prove Lemma \ref{rem_kmeans}.
\begin{proof}[Proof of Lemma \ref{rem_kmeans}]
For definiteness, we assume that $<v_i, \tilde{v}_i>$ is non-negative.
By (3.9) of the paper, it is easy to check that with $1-o(1)$ probability, we have 
\begin{equation}\label{key_representation}
\mu_i \tilde{v}_i=X^*X \tilde{v}_i+d^2_i\sqrt{a_2(d_i)}v_i+O(N^{-1/2+C\epsilon_0}),
\end{equation}
where we use the fact that $S$ is sparse and Markov inequality. When $k \in \mathcal{J}_i,$ assume that $|\tilde{v}_i(k)|\leq N^{-1/2+C\epsilon_0},$ using (\ref{key_representation}) and Markov inequality, we conclude that 
\begin{equation*}
\mu_i \tilde{v}_i(k)=d_i^2 \sqrt{a_2(d_i)}v_i(k)+O(N^{-1/2+C\epsilon_0}),
\end{equation*}
which is a contradiction. Hence, for all $k \in \mathcal{J}_i,$ we have $|\tilde{v}_i(k)|>N^{-1/2+C\epsilon_0}.$  When $k \notin \mathcal{J}_i, $ (\ref{key_representation}) reads as 
\begin{equation} \label{key_rep2}
\mu_i \tilde{v}_i(k)=(X^*X \tilde{v}_i)(k)+O(N^{-1/2+C\epsilon_0}),
\end{equation}
where we use Definition 2.1 of the paper.  Assume that $|\tilde{v}_i(k)|>N^{-1/2+C\epsilon_0},$ denote $X^J$ as the minor of $X$ by deleting the $j$-th columns with $j \in \mathcal{J}_i $  and  $\tilde{v}_{i}^{J}$ as the subvector of $\tilde{v}_i$ by deleting the entries with indices in $\mathcal{J}_i.$   As $|\mathcal{J}_i|=O(1),$ by (\ref{key_rep2}), with $1-o(1)$ probability, we have 
\begin{equation*}
\mu_i \tilde{v}^J_i(k)=((X^J)^*X^J \tilde{v}^J_i)(k)+O(N^{-1/2+C\epsilon_0}).
\end{equation*}
This yields that
\begin{equation*}
\frac{1}{||\tilde{v}^J_i||_2^2}(\tilde{v}_i^J)^*(X^J)^* X^J  \tilde{v}^J_i \rightarrow \mu_i.
\end{equation*}
Using Rayleigh quotient and the continuity of eigenvalues, when $N$ is large enough, we conclude that with $1-o(1)$ probability
\begin{equation*}
\lambda_1((X^J)^* X^J) \geq \mu_i,
\end{equation*} 
which is a contradiction by (2.8) of the paper. Here we use the fact 
 $(X^J)^* X^J$ is a $|\mathcal{J}^c_i| \times |\mathcal{J}^c_i|$ sample covariance matrix satisfying the MP law. Hence, for all $k \notin \mathcal{J}_i,$ we have $|\tilde{v}_i(k)|\leq N^{-1/2+C\epsilon_0}.$
\end{proof} 

\end{proof}

\begin{proof}[Proof of Theorem 3.5] We start with the proof of (1). The consistency of $\hat{\eta}_k$ is an immediate result of \cite[Theorem 2.9]{BGN}. For the convergent rate, by definition
\begin{equation*}
\eta_k=\sum_{k_1=1}^{k^+} d_{k_1}\mu_{k_1 k} \nu_{k_1 k}+\sum_{k_1=k^++1}^{r}d_{k_1}\mu_{k_1 k} \nu_{k_1 k}.
\end{equation*} 
Hence, the proof follows from Theorem 2.2 and 2.3 of the paper. Next, we prove (2).  Using a similar discussions to equations (3.9) and (3.10) of the paper, for some constant $C>0,$ with $1-o(1)$ probability, we have 
\begin{align*}
|| S-\hat{\mathcal{S}} ||_F^2= \sum_{k=1}^r d_k^2+\sum_{k=1}^q \hat{\eta}_k^2-2 \sum_{k=1}^q d_k \hat{\eta}_k\mu_{kk} \nu_{kk}+O(N^{-1/2+C\epsilon_0}),
\end{align*}
where we use Theorem 2.2 and 2.3 of the paper and (1) of Theorem 3.5. Therefore, the proofs come from part (1) of Theorem 3.5 and Proposition 3.3 of the paper. 
\end{proof}

\begin{proof}[Proof of Lemma 4.4]
(4.12) of the paper is from an elementary calculation.  For the proof of (4.13) and its monotonicity, choose any $ \ x > y > \lambda_{+},$ we have
\begin{equation*}
xm_{1c}(x)m_{2c}(x)-ym_{1c}(y)m_{2c}(y)=\frac{x-y-(g(x)-g(y))}{2c^{-1}}, 
\end{equation*}
where $ g(t):=\sqrt{(t+c^{-1}-1)^2-4c^{-1}t}.$ When $ t > \lambda_{+},$ we have
\begin{equation*}
g^{\prime}(t)=\frac{t-(c^{-1}+1)}{\sqrt{t^2+(c^{-1}-1)^2-2t(c^{-1}+1)}}>\frac{t-(c^{-1}+1)}{\sqrt{t^2+(c^{-1}+1)^2-2t(c^{-1}+1)}}=1,
\end{equation*}
where we need $t>\lambda_{+}$ to ensure the positiveness of $g(t)$. Hence, by the mean value theorem,  we conclude the proof.
\end{proof}

\begin{proof}[Proof of Lemma 4.6]
By an elementary computation on (4.13) of the paper , we have
\begin{equation*}
\mathcal{T}^{\prime}(p(d))=\frac{-1}{d^4-c^{-1}} \sim (d-c^{-1/4})^{-1}.
\end{equation*}
It is easy to check that there exists a constant $C>0,$ such that $| \mathcal{T}^{\prime \prime}(\xi) | \leq C$ for $\xi \in I_d.$ This concludes our proof by mean value theorem. 
  
\end{proof}

\begin{proof}[Proof of Lemma 4.8]
To prove (4.19) of the paper, we write
\begin{equation*}
\tilde{G}(z)=(H+\mathbf{U}\mathbf{D}\mathbf{U}^*-z)^{-1}.
\end{equation*}
The proof follows from the Woodbury matrix identity 
\begin{equation*}
(A+SBT)^{-1}=A^{-1}-A^{-1}S(B^{-1}+TA^{-1}S)^{-1}TA^{-1},
\end{equation*}
with $A=H-z, B=\mathbf{D}^{-1}, S=\mathbf{U}, T=\mathbf{U}^*.$ For the proof of (4.20), by (4.19) of the paper, we have
\begin{equation*}
\mathbf{U}^* \tilde{G}(z) \mathbf{U}=\mathbf{U}^*G(z)\mathbf{U}-\mathbf{U}^*G(z)\mathbf{U}(\mathbf{D}^{-1}+\mathbf{U}^*G(z)\mathbf{U})^{-1} \mathbf{U}^* G(z) \mathbf{U},
\end{equation*}
the proof  follows from the following identity
\begin{equation*}
A-A(A+B)^{-1}A=B-B(A+B)^{-1}B,
\end{equation*}
with $A=\mathbf{U}^*G(z)\mathbf{U}, B=\mathbf{D}^{-1}.$
\end{proof}

\begin{proof}[Proof of Lemma 5.3]
We first deal with (3.6). As $r$ is finite, we can choose a path $(\mathbf{d}(t): 0 \leq t \leq 1)$ connecting $\mathbf{d}(0)$ and $\mathbf{d}(1)$ having the following properties: \\

(i) For all $t \in [0,1]$, the point $\mathbf{d}(t) \in \mathcal{D}(\epsilon_0)$. \\

(ii) If $I_i^+(\mathbf{d}(1)) \cap I_j^+(\mathbf{d}(1))=\emptyset$ for a pair $1 \leq i<j \leq k^{+},$ then $I_i^+(\mathbf{d}(t)) \cap I_j^+(\mathbf{d}(t))=\emptyset$ for all $t \in [0,1].$ \\

Denote $\tilde{S}(t):=X+UD(t)V,$ where $D(t)$ is a diagonal matrix with elements $d_1(t), \cdots, d_r(t).$ As the mapping $t \rightarrow \tilde{S}(t)$ is continuous, we find that $\mu_{i}(t)$ is continuous in $t \in [0,1]$ for all $i,$ where  $\mu_{i}(t)$ are the eigenvalues of $\tilde{S}(t)\tilde{S}^*(t).$ Moreover, by Lemma 5.2 of the paper, we have 
\begin{equation} \label{persiible_t}
\bm{\sigma^+}(\tilde{S}(t)\tilde{S}^*(t)) \subset \Gamma(\mathbf{d}(t)), \forall \ t \in [0,1].
\end{equation}
In the case when the $k^+$ intervals are disjoint,  we have
\begin{equation*}
\mu_{i}(t) \in I_{i}^+(\mathbf{d}(t)), \ t \in [0,1],
\end{equation*}  
where we use property (ii) of the continuous path, (\ref{persiible_t}) and the continuity of $\mu_{i}(t).$ In particular, it holds true for $\mathbf{d}(1).$ Now we consider the case when they are not disjoint. Define $\mathcal{B}$ as a partition of $\{1, \cdots, k^+\} $ and denote the equivalent relation as
\begin{equation*}
i \equiv j  \ \   \text{if} \ \  I_{i}^+(\mathbf{d}(1)) \cap I_{j}^+(\mathbf{d}(1)) \neq \emptyset.
\end{equation*} 
Therefore, we can decompose $\mathcal{B}=\cup_{i} \mathcal{B}_i.$ It is notable that each $\mathcal{B}_i$ contains a sequence of consecutive integers.  Choose any $j \in \mathcal{B}_i, $ without loss of generality, we assume $j$ is not the smallest element  in $\mathcal{B}_i.$ Since they are not disjoint,  we have 
\begin{equation*}
 p^{\prime}(d_{j-1}^+)(d_{j-1}^+-d_j^+)\leq p(d_{j-1}^+)-p(d_{j}^+) \leq 2N^{-1/2+\epsilon_1+\epsilon_3}(d_{j-1}^+-c^{-1/4})^{1/2}, 
\end{equation*}
where we use the fact that $p^{\prime \prime}(x)>0$ and  (5.4) of the paper. This implies that 
\begin{equation*}
d^+_{j-1}-d^+_{j} \leq C N^{-1/2+\epsilon_1+\epsilon_3}(d_j^+-c^{-1/4})^{-1/2}, 
\end{equation*}
for some constant $C>0.$ By (5.5) of the paper,  we have
\begin{equation*}
(d_{j-1}^+-c^{-1/4})^{1/2} \leq (d_{j}^+-c^{-1/4})^{1/2}(1+\frac{d_{j-1}^+-d_j^+}{d_j^+-c^{-1/4}}) \leq (d_j^{+}-c^{-1/4})^{1/2}(1+o(1)).
\end{equation*}
Therefore, by repeating the process for the remaining $j \in \mathcal{B}_i,$ we find 
\begin{equation*}
\text{diam} (\cup_{j \in \mathcal{B}_i} I_{j}^+(\mathbf{d}(1))) \leq C N^{-1/2+C\epsilon_0} \min_{j \in \mathcal{B}_i}(d_j^+(1)-c^{-1/4})^{1/2}(1+o(1)), 
\end{equation*} 
where we use the fact that $r=O(1).$ This immediately yields that 
\begin{equation*}
| \mu_{i}(1)-p(d_i^{+}(1)) | \leq N^{-1/2+C\epsilon_0}(d_i^{+}(1)-c^{-1/4})^{1/2},
\end{equation*}
for some constant $C>0.$ This completes the proof of (3.6) of the paper. Finally, we deal with the extremal non-outlier eigenvalues (3.7) of the paper. 
By the continuity of $\mu_{i}(t)$ and Lemma 5.2 of the paper, we have
\begin{equation}\label{permissible_t2}
\bm{\sigma^0}(\tilde{S}(t)\tilde{S}^*(t)) \subset I^0(t), \ t \in [0,1].
\end{equation}
In particular it holds true for $\mathbf{d}(1).$  This concludes our proof. 
\end{proof}

\begin{proof}[Proof of Lemma 5.5]
A crucial estimate is the following lemma, which can be found in \cite[Lemma 5.6]{BKYY}. Define the boundary of $B_{\rho_k}(d_k)$ as $\partial B_{\rho_k}(d_k),$ then we have
\begin{lem} Denote
\begin{equation*}
\Gamma_k=\Gamma \cap \partial B_{\rho_k}(d_k),
\end{equation*}
then for $k \in A$, and $\zeta \in \Gamma_k$,  recall (5.13) of the paper, we have
\begin{equation} \label{sij2differencebound}
|\zeta-d_l| \sim \rho_k+|d_k-d_l|, \ 1 \leq l \leq r.
\end{equation}
\end{lem}

\noindent By  (4.9), (5.12) of the paper and the fact $r$ is finite, it is easy to check that
\begin{equation} \label{sij2boundinitial}
|S^{(2)}| \leq \int_{\Gamma} \frac{d^2_i d^2_j N^{-1+2\epsilon_1}}{|\zeta-d_i||\zeta-d_j|} \left | \left | \frac{1}{\mathbf{D}^{-1}+\mathbf{U}^*G(p(\zeta))\mathbf{U}} \right | \right | |d \zeta|. 
\end{equation}
We now assume  $\zeta \in \Gamma_k,$ by the resolvent expansion,  we have
\begin{align} \label{sij2resolvent}
(\mathbf{D}^{-1}&+\mathbf{U}^*G(p(\zeta))\mathbf{U})^{-1}=  (\mathbf{D}^{-1}+\mathbf{U}^*\Pi(p(\zeta))\mathbf{U})^{-1}  \nonumber \\
+ & (\mathbf{D}^{-1}+\mathbf{U}^*\Pi(p(\zeta))\mathbf{U})^{-1}\Delta(p(\zeta))(D^{-1}+\mathbf{U}^*G(p(\zeta))\mathbf{U})^{-1}.
\end{align}  
By (5.12) of the paper, we have
\begin{equation} \label{sij2bound1}
\left | \left | \Delta(p(\zeta)) \right | \right | \leq |p(\zeta)-\lambda_+|^{-1/4} N^{-1/2+\epsilon_1} \leq (d_k-c^{-1/4})^{-1/2}N^{-1/2+\epsilon_1}.
\end{equation}
For $1 \leq l \leq r, $  by Lemma 4.6 and (4.9) of the paper,  there exists some constant $\delta>0$, such that
\begin{equation}\label{sij2bound2}
|\mathcal{T}(p(\zeta))-d_l^{-2}| \geq  \delta (|\zeta-d_l| \wedge c^{-1/4}) \geq \delta|\zeta-d_k|=\delta\rho_k \geq \delta(d_k-1)^{-1/2}N^{-1/2+\epsilon_0},
\end{equation}
where in the last step we use (5.14) of the paper. Hence,  by (5.20) of the paper, (\ref{sij2resolvent}) and (\ref{sij2bound1}), we have
\begin{equation} \label{sij2matrixnormbound}
\left | \left | (\mathbf{D}^{-1}+\mathbf{U}^*G(p(\zeta))\mathbf{U})^{-1} \right | \right | \leq \frac{C}{\rho_k}. 
\end{equation}
 Decomposing $\Gamma$ into $\Gamma=\cup_{k \in A} \Gamma_k$, by (\ref{sij2differencebound}), (\ref{sij2boundinitial}), (\ref{sij2matrixnormbound}) and the fact $\Gamma_k$ has length $2\pi \rho_k$, we have
\begin{equation} \label{sij2secondbound}
|S^{(2)}| \leq C \sum_{k \in A} \sup_{\zeta \in \Gamma_k} \frac{d^2_i d^2_jN^{-1+2\epsilon_1}}{|\zeta-d_i||\zeta-d_j|} \leq C \sum_{k \in A} \frac{d^2_i d^2_jN^{-1+2\epsilon_1}}{(\rho_k+|d_k-d_i|)(\rho_k+|d_k-d_j|)},
\end{equation}   
for some constant $C>0.$ To estimate the right-hand side of  (\ref{sij2secondbound}),  for $i \notin A,$ by (5.14) of the paper,  we have that
\begin{equation*}
\rho_k \leq d_{k}-c^{-1/4} \leq |d_k-c^{-1/4}+c^{-1/4}-d_i| \leq |d_k-d_i|,
\end{equation*}
from which we conclude
\begin{equation*}
\sum_{k \in A} \frac{1}{(\rho_k+|d_k-d_i|)^2} \leq \sum_{k \in A} \frac{1}{|d_k-d_i|^2} \leq \frac{C}{\nu_i^2(A)}, 
\end{equation*}
where we use the fact that $r$ is finite. Similarly,  for $i \in A,$  we have $|d_k-d_i| \leq \rho_k$.  Combining with the fact $\rho_k+|d_i-d_k| \geq \rho_i$ for all $k \in A$, we have
\begin{equation*}
\sum_{k \in A} \frac{1}{(\rho_k+|d_k-d_i|)^2} \leq \frac{C}{\rho_i^2} \leq \frac{C}{\nu_i^2(A)}+\frac{C}{(d_i-c^{-1/4})^2},
\end{equation*}
for some constant $C>0. $
Combine with (\ref{sij2secondbound}),  we have
\begin{equation}\label{sij2boundfinal}
|S^{(2)}| \leq C d^2_i d^2_jN^{-1+2\epsilon_1}(\frac{1}{\nu_i}+\frac{\mathbf{1}(i \in A)}{|d_i-c^{-1/4}|})(\frac{1}{\nu_j}+\frac{\mathbf{1}(j \in A)}{|d_j-c^{-1/4}|}),
\end{equation} 
for some constant $C>0.$
\end{proof}

\begin{proof}[Proof of Lemma 5.6]
In the case $|d_i-c^{-1/4}|\leq \frac{1}{2},$ we have that (see the equation above (6.11) of \cite{BKYY})
\begin{equation*}
|1-d_i^2 \mathcal{T}(z)| \geq d_i^2[\max\{|d_i^{-2}-c^{1/2}|-|\operatorname{Re} \mathcal{T}(z)-c^{1/2}|, 0\}+\operatorname{Im}\mathcal{T}(z)].
\end{equation*}
By \cite[(6.11)]{BKYY}, we have that for any $y \leq tz, \ t \geq 1, $ 
\begin{equation}\label{basicinequality}
\max\{x-y,0\}+z \geq \frac{x}{3t}+\frac{z}{3}.
\end{equation}
For $\mu_j \in [\lambda_{-}, \ \lambda_+],$ by Lemma 4.5 and (5.31) of the paper,  using $t=C$ in (\ref{basicinequality}), we find that there exists some constant $\delta>0,$ such that
\begin{equation*}
|1-d_i^2 \mathcal{T}(z)| \geq \delta d_i^2\left( |d_i^{-2}-c^{1/2}|+\operatorname{Im}\mathcal{T}(z) \right).
\end{equation*}
When $\mu_j \in [\lambda_+,\ \lambda_++N^{-2/3+C\epsilon_0}],$ choosing $t=K^{C\epsilon_0}$ in (\ref{basicinequality}) and using (5.31) of the paper, we get 
\begin{equation*}
|1-d_i^2 \mathcal{T}(z)| \geq \delta d_i^2\left(|d_i^{-2}-c^{1/2}|+\operatorname{Im}\mathcal{T}(z) \right).
\end{equation*}
When  $|d_i-c^{-1/4}|\geq \frac{1}{2},$ by Lemma 4.5 of the paper, for $\mu_j \in [\lambda_{-}, \lambda_{+}+N^{-2/3+C\epsilon_0}],$ we have
\begin{equation*}
|1-d_i^2\mathcal{T}(z)| \geq  \delta d_i^2(|d_i^{-2}-c^{1/2}|+\operatorname{Im} \mathcal{T}(z)).
\end{equation*}  

\end{proof}

\end{document}